\documentclass[]{interact}

\usepackage[english]{babel}
\usepackage{amsmath,amsthm}
\usepackage{amsfonts}
\usepackage{dsfont}

\usepackage{epstopdf}
\usepackage[caption=false]{subfig}

\usepackage[numbers,sort&compress]{natbib}
\bibpunct[, ]{[}{]}{,}{n}{,}{,}
\renewcommand\bibfont{\fontsize{10}{12}\selectfont}

\theoremstyle{plain}
\newtheorem{thm}{Theorem}[section]
\newtheorem{lem}[thm]{Lemma}
\newtheorem{cor}[thm]{Corollary}
\newtheorem{prop}[thm]{Proposition}

\theoremstyle{definition}
\newtheorem{defn}[thm]{Definition}

\theoremstyle{remark}
\newtheorem{rem}{Remark}

 \newcommand\R{\mathbb{R}}

 \newcommand\E{\mathbb{E}}

 \DeclareMathOperator*{\var}{Var}
 \DeclareMathOperator*{\cov}{Cov}
 
 \DeclareMathOperator*{\esssup}{ess\;sup}
 \newcommand\tab[1][0.70cm]{\hspace*{#1}}
\begin{document}

\articletype{RESEARCH ARTICLE}

\title{On Besov regularity and local time of the stochastic heat equation}

\author{
\name{Brahim Boufoussi\textsuperscript{a}\thanks{CONTACT Brahim Boufoussi Email: boufoussi@uca.ac.ma} and Yassine Nachit\textsuperscript{b}}
\affil{\textsuperscript{a,b}Department of Mathematics, Faculty of Sciences Semlalia, Cadi Ayyad University, 2390 Marrakesh, Morocco}
}

\maketitle

\begin{abstract}
Sharp Besov regularities in time and space variables are investigated for $\left(u(t,x),\; t\in [0,T],\; x\in \R\right)$, the mild solution to the stochastic heat equation driven by space-time white noise. Existence, H\"{o}lder continuity, and Besov regularity of local times are established for $u(t,x)$ viewed either as a process in the space variable or time variable. Hausdorff dimensions of their corresponding level sets are also obtained.
\end{abstract}

\begin{keywords}
Stochastic heat equation; White noise; Besov-Orlicz spaces; Schauder functions; Haar basis; Local times; H\"{o}lder continuity; Hausdorff  dimension
\end{keywords}

\begin{amscode}
60G15; 60G17; 60H05; 60H15
\end{amscode}

\section*{Introduction}
Stochastic partial differential equations (SPDEs) model various random phenomena in physics, finance, fluid mechanics,  among others, see, e.g. \cite{Kou, Denk, Carmona}.
They have been widely investigated by different approaches in the last three decades. Let us mention the analytic method \cite{Krylov 1977, Krylov 1982a, Krylov 1982b}, the semigroup point of view \cite{Da Prato}, and the probabilistic setting using the theory of martingale measures \cite{Walsh}. The particular case of stochastic heat equations has been intensively studied from many perspectives: See, e.g., \cite{BalanTudor1, SoleSara1, SoleSara, DalangSole, ArayaTudor, XiaYan} for regularity investigations and \cite{ChenDalang, CONUS, Mueller} for many other studies.
The regularity in Besov spaces of SPDEs has received much attention in the past decade. Among other things, such studies are closely related to the theme of adaptive numerical wavelet methods. For more details on this subject, we refer to \cite{Cioica1, Cioica2, Cioica3}.
\par
In this paper, we consider the following linear stochastic heat equation
\begin{equation}\label{EQ0}
    \begin{split}
       \frac{\partial u}{\partial t}(t,x) &= \frac{1}{2}\Delta u(t,x)+\frac{\partial^2}{\partial t\partial x}W(t,x),\quad t>0,\;x\in \R, \\
         u(0,x) &= 0,\qquad x\in \R,
    \end{split}
\end{equation}
where $\frac{\partial^2}{\partial t\partial x}W(t,x)$ is a space-time white noise. In Section \ref{Linear stochastic heat equation}, following Walsh's random field approach \cite{Walsh}, we will briefly give a rigorous formulation of the formal equation \eqref{EQ0}. The purpose of this paper is to investigate the Besov regularity for the process $(u(t,x),\;t\in[0,T],\;x\in [a,b])$ as well as its local time, respectively in the space variable $x$ (for fixed $t$) and in the time variable $t$ (for fixed $x$), with $T>0$ and $a<b$ are arbitrary real values.\par
Besov spaces, involving a bounded or unbounded interval $I\subseteq \R$, usually noted in the literature  by
$ \mathcal{B}^{\alpha}_{p, q}(I) $ with $ 0<\alpha<1 $, $ 1\leq p, q\leq +\infty $, are a set of functions of $L^p(I)$ having a smoothness of order $\alpha$.
They cover some classical function spaces as special cases. Namely,  $\mathcal{B}^{\alpha}_{p, p}(\R) $ coincides with the classical Sobolev space $W^{\alpha, p}(\R) $ and
$\mathcal{B}^{\alpha}_{\infty, \infty}(I) $ is the classical $\alpha$-H\"{o}lder space
${\mathcal{H}}^{\alpha}(I)$. When the Orlicz norm is used in place of the $L^p$-norm, we obtain the well-known Besov-Orlicz spaces. \\
In this paper, we are essentially concerned by  the class of Besov spaces $ \mathcal{B}^{\alpha}_{p, \infty}([0,1]) $ and also by Besov-Orlicz spaces  $ \mathcal{B}^{\alpha}_{{\mathcal{N}}, \infty}([0,1]) $, $\mathcal N$ is  the Young function $\mathcal{N}(x)= e^{x^2}-1$.  We will note these spaces respectively by $ \mathcal{B}^{\alpha}_{p}$ and $ \mathcal{B}^{\alpha}_{\mathcal N}$. In this case, for $\alpha p>1$ and for any
$\varepsilon>0$, we have the following continuous injections:
\begin{equation}\label{continj}
{\mathcal{H}}^{\alpha+\varepsilon}\hookrightarrow
 \mathcal{B}^{\alpha, 0}_{p} \hookrightarrow \mathcal{B}^{\alpha}_{p} \hookrightarrow \mathcal{H}^{\alpha -1/p},
 \end{equation}
where $ \mathcal{B}^{\alpha, 0}_{p}  $ is a separable subspace of
$\mathcal{B}^{\alpha}_{p}$. Furthermore, it is important to recall that the Besov-Orlicz space $ \mathcal{B}^{\alpha}_{{\mathcal{N}}} $ is continuously embedded in $\mathcal{B}^{\alpha}_{p}$ for any $p\geq1$. See section \ref{BesovOrliczSection} for more details and \cite{Tribel} for an introduction to Besov spaces.\\
Our first main result is to prove that, for any $p\geq4 $ we have: \par\noindent
i) For any fixed space variable $x$,
$$\mathbb{P}\left[(u(t, x))_{t\in[0,1]}\in \mathcal{B}^{1/4}_{p}\right]=1\,\,\,\text{and}\,\,\,
\mathbb{P}\left[(u(t, x))_{t\in[0,1]}\in \mathcal{B}^{1/4, 0}_{p}\right]=0 .$$
ii) For any fixed time variable $t$,
$$\mathbb{P}\left[(u(t, x))_{x\in[0,1]}\in \mathcal{B}^{1/2}_{p}\right]=1\,\,\,\text{and}\,\,\,
\mathbb{P}\left[(u(t, x))_{x\in[0,1]}\in \mathcal{B}^{1/2, 0}_{p}\right]=0 .$$
In fact, we even get a better result showing that i) and ii) are true respectively in the  Besov-Orlicz space $ \mathcal{B}^{1/4}_{\mathcal N}$ and $ \mathcal{B}^{1/2}_{\mathcal N}$.
We should point out here that, by a suitable affine  change of variables, the interval $[0,1]$ may be replaced in i) (resp. ii)) by any arbitrary compact interval $[0,T]$ (resp. $[a,b]$). \noindent\par
On the other hand,  injections (\ref{continj}) show that the obtained regularity results are the best one can get in the scale of Besov spaces. They improve the classical H\"{o}lder regularities of $u(t,x)$; namely, the process $u(t,x)$ satisfies a.s. $\alpha$- H\"{o}lder condition with $\alpha<\frac{1}{4}$ (resp. $\alpha<\frac{1}{2}$) in the time (resp. space) variable, see e.g. \cite{DalangSole}, \cite{SoleSara} and references therein.\par\noindent
As far as we know, only the spatial Besov regularity of $u(t,x)$
has been partially investigated firstly in Deaconu’s thesis \cite{Deaconu}, where the author has also claimed the temporal regularity as in i) but she failed to provide a proof for her finding.  Our approach is certainly much more technical than that of \cite{Deaconu}, but it allowed us to obtain sharp results. Our proof is based on the characterization of Besov spaces in terms of sequences spaces (Theorem \ref{Besov}). We use essentially the same arguments, with some adjustments, like those used by  Ciesielski, Kerkyacharian, and Roynette in \cite{Roynette}, where the authors have investigated Besov regularity for a large class of Gaussian processes. Recently, a more direct method, which uses the usual modulus-of-continuity definition of the Besov norms, has been employed in \cite{Verrar1} to prove the temporal regularity in the Besov-Orlicz space for solutions to parabolic stochastic differential equations.
Many other works investigating different problems have been published using the sequential characterization of Besov topology, we refer e.g., to
this non-exhaustive list \cite{RoynetteMB}, \cite{Roynette}, \cite{BLD}, \cite{Boufoussi}.
\noindent\par

Our second aims are to investigate existence,  H\"{o}lder  continuity, and Besov regularity of local times of
$ u(t,x)$, viewed as a process respectively in time and space variables.
We will use the notion of local nondeterminism (LND), initiated by Berman in  \cite{Berman2} and extended later  in \cite{Cuzick} to the strong local nondeterminism concept (SLND). These two notions are the most important mathematical tools used by several authors to study  sample path properties for various Gaussian processes and random fields, see, e.g. \cite{Xiao1, Xiao2, Xiao3, Xiao4, Xiao5, AyaWuXiao, BDG} and references therein.\noindent\par
We must note  that, for fixed $x$, the process  $(u(t,x), t\in[0,T])$ is identically distributed as the so-called bifractional Brownian motion (see \cite{Sw}). So, many sample path properties of the process $(u(t,x), t\in[0,T]) $, including the H\"{o}lder continuity of its local time, may be deduced from \cite{TudorXiao}. In this  last paper, to prove  that the bifractional Brownian motion satisfies the SLND property, the authors have used the Lamperti transformation to connect self-similar and stationary Gaussian processes.
In our paper, to verify the LND condition for the processes $(u(t,x), t\in[0, T], x\in\R) $ respectively in time and space variables, we use fine estimates on the Green kernel $G$ and careful calculations.
Furthermore, we will improve the Knowledge of the sample paths properties for these two processes by proving regularity results of their local times in the modular Besov spaces $\mathcal{B}^{\omega}_{p}$, with respect to modulus $\omega(t)=t^{\alpha}(\log{1/t})^{\lambda} $, with $\alpha=1/4$ (resp. $1/2$) and $p\lambda>1$.
In contrast with the Besov regularity obtained for the Brownian local time in \cite{BoufoussiR}, our results are not optimal, but they have the merit of being new and sharper than what is known. \\

Let us give an overview of the results existing in the literature related to the context of our paper.
First, it was raised in \cite{Xiao4} the relevance of studying the sample paths properties for the solutions to stochastic heat or wave equations through their local times and LND concept.
Ouahhabi and Tudor \cite{Ouahhabi} have studied, for fixed space variable, the H\"{o}lder regularity of the local time of the solution to the $d$-dimensional
stochastic linear heat equation driven by a fractional-white noise, with Hurst parameter $H>\frac{1}{2}$. The investigation of the sample paths properties, both in time and space variables, of the solution to the stochastic heat equation driven by a fractional-colored  noise is considered in \cite{Tudor} and  \cite{Xiao5}.

While writing this article, we discovered the papers \cite{Verrar2, Verrar3, Verrar1} where a direct method was used to handle Besov norms. They gave alternative proofs of Besov regularities of Brownian motion and fractional Brownian motion, without the equivalent wavelet description used in \cite{Ciesielski} and \cite{Roynette}.
By this new approach, we expect that we will establish sharp Besov regularities for the solution and its local time to a d-dimensional stochastic heat (or wave) equation driven by a fractional-colored noise. We intend to study these questions in a future paper.

This paper is organized as follows: In the first paragraph, we briefly recall some notions on  local times, Besov and Besov-Orlicz spaces. The second paragraph is devoted to studying  the Besov-Orlicz regularity for the sample paths of $t\mapsto u(t,x)$ and $x\mapsto u(t,x)$. In the third paragraph, we investigate the existence and Besov regularity of local times of the process $u(t,x)$ in time and space variables and conclude by considering Hausdorff dimensions of the level sets $\{t\in [0,T];\;u(t,x)=u(t_0,x)\}$ and
$\{x\in [a,b];\;u(t,x)=u(t, x_0)\}$.

\section{Preliminaries}
In this paragraph, we give some basic notions on local times and Besov spaces.
\subsection{The local times}
We recall here the concept of the local time, in the Fourier analytic sense, as it has been initiated by S. Berman in \cite{Berman1}.\\
Let $I\subset\R$ be a compact interval, and  $\theta: s\in I\to \theta_s\in \R$, a deterministic Borel function. $\mathcal{B}(I)$ be the Borel $\sigma$-algebra on $I$, for any $B\in \mathcal{B}(I)$  we define the occupation measure $\mu_B$ by
$$\mu_B(A)=\lambda\{s\in B,\; \theta_s\in A\},\quad A\in \mathcal{B}(\R).$$
If $\mu_B$ is absolutely continuous with respect to the Lebesgue measure on $\R$, we say that $\theta$ has a local
time on $B$, and we define its local time $L(.,B)$  as the Radon Nikodym derivative of $\mu_B$ with respect to the Lebesgue measure. If $B=[0,t]$ (resp. $B=I$) we write $L(\xi,t)$ (resp. $L(\xi)$) instead of $L(.,[0,t])$ (resp. $L(\xi,I)$). This definition can be extended to any measurable and bounded (or positive) function $f$ to get the so-called occupation density formula:
$$\int_{0}^{t}f(\theta_s)ds=\int_{\R}f(\xi)L(\xi,t)d\xi.$$
The idea of Berman is to relate properties of $L(.,B)$ with the integrability of the Fourier
transform of $\theta$. Recall the following essential result:
\begin{prop}
  The function $\theta$ has a square integrable local time $L(\xi,B)$ iff
  $$ \int_{\R}\left|\int_{B}\exp(iu\, \theta_s)ds\right|^2du<\infty.$$
  Moreover, we have the following representation of the local time, for almost every $\xi \in \R$,
  \begin{equation}\label{local time representation as inversion formula}
  L(\xi,B)=\frac{1}{2\pi}\int_{\R}\int_{B}\exp(iu(\theta_s-\xi))dsdu.
\end{equation}
\end{prop}

The deterministic function $\theta$ can be chosen to be a sample path of a stochastic process $(X_t,\;t\in [0,T])$. To prove almost sure existence and square integrability of the local time $L(\xi,B)$, it is enough to establish that
$$\E \int_{\R}\left|\int_{B}\exp(iuX_s)ds\right|^2du<\infty.$$
When $(X_t,\;t\in [0,T])$ is Gaussian, then we get the well-known Berman's criterion:
\begin{prop}[\cite{Berman1}]
  If $(X_t,\;t\in [0,T])$ is a centred Gaussian process, and satisfies
  $$\int_{0}^{T}\int_{0}^{T}\left[\E(X_s-X_t)^2\right]^{-1/2}dsdt<\infty.$$
  Then for almost all $\omega$, for any $B\in \mathcal{B}([0,T])$, the trajectory $s\mapsto X_s(\omega)$   has a local time $L(\xi, B, \omega)$ which is square integrable with respect to $\xi$.
\end{prop}
\begin{rem}
   In the rest of this paper, we will note the local time of a process by $L(\xi, B)$ instead of $L(\xi, B, \omega)$ unless a confusion exists.
\end{rem}
The following theorem will clarify Berman’s principle, providing the link between the regularity of the local time and  the irregularity of its underlying stochastic process:
\begin{thm}[\cite{Berman1}]\label{thmberman}
Let $(X_t,\;t\in [0,T])$ be a centred Gaussian process, such that for some $p\geq 0$,
\begin{equation}\label{(p+1)/2}
  \int_{0}^{T}\int_{0}^{T}[\E(X_s-X_t)^2]^{-(p+1)/2}dsdt<\infty.
\end{equation}
Then for almost all $\omega$, the trajectory  $s\mapsto X_s(\omega)$ has a local time $L(\xi)$, such that the first $[p/2]$
derivatives of $L(\xi)$ exist and are square integrable ($[v]$ = integral part of v).
Moreover, when $[p] $ is even, the sample functions of $X$ are nowhere $(2/(p + 1))$-H\"{o}lder continuous.
\end{thm}
\noindent
\begin{rem}
A particular situation is obtained if,
$$\E(X_t-X_s)^2 \geq c |t-s|^{\beta},$$
where $0 < \beta < 2$ and  $c$ is a positive constant.  We can easily deduce that a.s. the trajectories of the process $X$ are  nowhere $\beta$-H\"{o}lder continuous.\par
\end{rem}
To derive from Kolmogorov continuity theorem, the joint continuity of the local time and the H\"{o}lder continuity in $t$ of $L(\xi,t)$, our starting point is the following identities about the moments of the increments of the local time
\begin{equation*}
\begin{split}
  & \E[L(\xi+k,t+h)-L(\xi,t+h)-L(\xi+k,t)+L(\xi,t)]^n \\
   =& (2\pi)^{-n}\int_{[t,t+h]^n}\int_{\R^n}\prod_{j=1}^{n}(e^{-i(\xi+k)u_j}-e^{-i\xi u_j})\E[\exp(i\sum_{j=1}^{n}u_jX_{t_j})]d\overline{u}d\overline{t},
\end{split}
\end{equation*}
and
\begin{equation*}
  \E[L(\xi,t+h)-L(\xi,t)]^n = (2\pi)^{-n}\int_{[t,t+h]^n}\int_{\R^n}\exp(-i\xi\sum_{j=1}^{n}u_j)\E[\exp(i\sum_{j=1}^{n}u_jX_{t_j})]d\overline{u}d\overline{t},
\end{equation*}
where $\overline{u}=(u_1,...,u_n)$, $\overline{t}=(t_1,...,t_n)$, $\xi$, $\xi+k\in \R$ and $t$, $t+h\in [0,T]$.
To find suitable upper bounds for these increments for a given centered Gaussian process $ X $, the expression $\var[\sum_{j=1}^{n} v_j(X_{t_j}-X_{t_{j-1}})]$ appearing by an appropriate change of variables in the characteristic function above, is lower bounded by $C_m\sum_{j=1}^{n} v_j^2 \var(X_{t_j}-X_{t_{j-1}})$ when $X$ verifies the LND property (see, \cite[Lemma 2.3.]{Berman2}). This last property reads as follows:
\begin{equation}\label{LND}
  \lim_{c\to 0}\inf_{0\leq t-r\leq c,\;r<s<t}\frac{\var(X_t-X_s|X_{\tau},\;r\leq \tau \leq s)}{\var(X_t-X_s)}>0.
\end{equation}
To analyze the local time of the solution to the stochastic heat equation, the LND property will be verified in the next for $u(t,x)$ viewed either as a process in time variable for any fixed $x\in \R$ or a process in space variable for any fixed $t$.

\subsection{Modular Besov spaces}\label{Besovsection}
\tab
In this paragraph, we recall some notions of modular Besov spaces and their characterizations in terms of some sequences spaces. Let $I\subset\R$ be a compact interval, $1\leq p<\infty$ and $f\in L^p(I\,;\,\R)$. We can measure the smoothness of $f$ by its modulus of continuity computed in the $L^p$-norm. For this end, let us define for any $t>0$,
$$\Delta_{p}(f,I)(t)=\sup_{|s|\leq t}\left\{\int_{I_s}|f(x+s)-f(x)|^pdx\right\}^{\frac{1}{p}},$$
where $I_s=\{x\in I;\;\;x+s\in I\}$.\\
We can find in the literature various ways to define Besov norms.  They are generally defined with respect to some moduli, which are non-decreasing and continuous positive functions $\omega$ defined on $[0,+\infty[$, s.t. $\omega(0)=0$.
The most typical example of moduli is:
\begin{equation}\label{Modulus}
\omega_{\alpha, \lambda}(t)=t^\alpha(\log(1/t))^{\lambda} ,
\end{equation}
 with $0<\alpha<1$ and $\lambda\in \R$. We refer to \cite{Roynette}  for a more details on this subject.
 \begin{rem}
   From now on, we will only be concerned with the moduli $\omega_{\alpha, \lambda}$ defined in \eqref{Modulus} and write $\omega$ instead of $\omega_{\alpha, \lambda}$ unless a confusion exists.
 \end{rem}
Let $\omega$ be any modulus defined by \eqref{Modulus}, we consider the norm
$$||f||_{\omega, p}:=||f||_{L^p(I)}+\sup_{0< t\leq 1}\frac{\Delta_{p}(f,I)(t)}{\omega(t)}.$$
The modular Besov space is given by
$$\mathcal{B}^{\omega}_{p}(I)=\{f\in L^p(I);\;\; ||f||_{\omega, p}<\infty\}.$$
The space $(\mathcal{B}^{\omega}_{p}(I), ||\, .\, ||_{\omega, p})$ is a non separable Banach space. We also define
$$\mathcal{B}^{\omega,0}_{p}(I)=\{f\in L^p(I);\;\; \Delta_{p}(f, I)(t)=o(\omega(t))\;\text{as}\;t\rightarrow 0^+\}.$$
$\mathcal{B}^{\omega,0}_{p}(I)$ is a separable subspace of $\mathcal{B}^{\omega}_{p}(I)$.
For $p=\infty$, the space $\mathcal{B}^{\omega}_{\infty}(I) $ is defined in the same way by using the usual $L^\infty$-norm. In this case it coincides with
 $\mathcal{H}^{\omega}(I)$, the $\omega$-H\"{o}lder space defined by
    \begin{equation}\label{C omega}
      \mathcal{H}^{\omega}(I):=\left\{f\,:I\longrightarrow\R\;\; \vert\,\,\,\esssup_{x,y\in I \atop x\neq y }\frac{|f(x)-f(y)|}{\omega(|x-y|)}<\infty\right\},
    \end{equation}
    endowed with the norm $||f||_{\omega,\infty}=\displaystyle \esssup_{x\in I}|f(x)|+\esssup_{x,y\in I \atop x\neq y }\frac{|f(x)-f(y)|}{\omega(|x-y|)}.$

As it is mentioned in the introduction, standard  changes of variables permit  us to restrict ourselves to the unit interval $I=[0,1] $, so we will omit to specify the interval $I$ in our notations.
The following theorem (Theorem \ref{Besov})  is a characterization of modular Besov spaces $\mathcal{B}^{\omega}_{p}$ in terms of progressive differences of a function in dyadic points.
Its proof has been established for general moduli in  \cite{Roynette}.
\begin{thm}\label{Besov}
Let $\omega$ be as in \eqref{Modulus}, $p>1$, $\frac{1}{p}<\alpha<1$ and $\lambda\geq 0 $, then we have
\begin{enumerate}
  \item  $\mathcal{B}^{\omega}_{p}$ is linearly isomorphic to a sequences space and we have the following equivalence of norms:
      $$||f||_{\omega, p}\sim\max\left\{|f_0|,|f_1|,\sup_{j}\frac{2^{-j\left(\tfrac{1}{2}+\tfrac{1}{p}\right)}}{\omega(2^{-j})}\left[\sum_{k=1}^{2^{j}}|f_{jk}|^p\right]^{\tfrac{1}{p}} \right\},$$
where the coefficients $\left\{f_0,  f_1, f_{jk}\,, j\geq0\,, 1\leq k\leq 2^{j}\right\}$ are given by
$$f_0=f(0),\qquad f_1=f(1)-f(0),$$
$$f_{jk}=2\cdot 2^{j/2}\left\{f\left(\frac{2k-1}{2^{j+1}}\right)-\frac{1}{2}f\left(\frac{2k}{2^{j+1}}\right)-\frac{1}{2}f\left(\frac{2k-2}{2^{j+1}}\right)\right\}.$$
  \item $f$ is in $\mathcal{B}^{\omega, 0}_{p}$ if and only if
  $$\lim_{j\to 0}\frac{2^{-j\left(\tfrac{1}{2}+\tfrac{1}{p}\right)}}{\omega(2^{-j})}\left[\sum_{k=1}^{2^j}|f_{jk}|^p\right]^{\tfrac{1}{p}}=0.$$
\end{enumerate}
\end{thm}
When $\omega(t)=t^{\alpha}$, we will use the notations:
$$||f||_{\alpha,p}:=||f||_{\omega, p},\;\;\;\mathcal{B}^{\alpha}_{p}:=\mathcal{B}^{\omega}_{p}\;\;\;\text{and}\;\;\;\mathcal{B}^{\alpha, 0}_{p}:=\mathcal{B}^{\omega,0}_{p}.$$
\begin{rem} One can easily show, by Theorem \ref{Besov}, the following useful continuous injections
\begin{itemize}
\item For any $\varepsilon>0$, $1\leq p<\infty$ and $\frac{1}{p}<\alpha<1$,
 $$\mathcal{H}^{\alpha+\varepsilon}\hookrightarrow\mathcal{B}^{\alpha, 0}_{p}\hookrightarrow\mathcal{B}^{\alpha}_{p}\hookrightarrow \mathcal{H}^{\alpha-\frac{1}{p}}. $$
  \item Let $1\leq p<\infty$ and $0<\alpha<\beta<1$, we have
    $$\mathcal{B}^{\beta}_{p}\hookrightarrow \mathcal{B}^{\alpha, 0}_{p}.$$

\item Let $\omega(t)=t^{\alpha} (\log{(1/t)})^{\lambda}$, $0<\alpha<1$ and $\lambda\geq 0 $, then
\begin{equation}\label{Rinject}
\mathcal{B}^{\omega}_{p}\hookrightarrow \mathcal{H}^{\gamma},
\end{equation}
for any $\gamma<\alpha $ and $p$ sufficiently large.
\end{itemize}
\end{rem}
\subsection{Besov–Orlicz spaces}\label{BesovOrliczSection}
Let $I\subset \R$, be a compact interval, and $\mathcal N$ is the Young function defined by $\mathcal{N}(x)=e^{x^2}-1$. The Orlicz space $L_{\mathcal{N}}(I)$ is the space of measurable  functions $f:I\mapsto \R$, such that
$$\|f\|^*_{\mathcal{N}}:=\inf_{\lambda>0}\frac{1}{\lambda}\left[1+\int_{I}\mathcal{N}(\lambda f(t))dt\right]<\infty.$$
It is more suitable to use an equivalent norm to $\|\cdot\|^*_{\mathcal{N}}$ (see e.g. Ciesielski \cite{Ciesielski}):
$$\|f\|_{\mathcal{N}}=\sup_{p\geq 1}\frac{\|f\|_{L^p(I)}}{\sqrt{p}}.$$
Let $\Delta_{\mathcal{N}}(f, I)(t)$ be the modulus of continuity of $f$ in the Orlicz space $L_{\mathcal{N}}(I)$ defined as:
$$\Delta_{\mathcal{N}}(f,I)(t)=\sup_{p\geq 1}\frac{\Delta_p(f,I)(t)}{\sqrt{p}}.$$
For $0<\alpha<1$, we consider the following norm
$$\|f\|_{\alpha, \mathcal{N}}=\|f\|_{\mathcal{N}}+\sup_{0<t\leq 1}\frac{\Delta_{\mathcal{N}}(f,I)(t)}{t^{\alpha}}.$$
The Besov-Orlicz space is defined by
$$\mathcal{B}^{\alpha}_{\mathcal{N}}(I):=\{f\in L_{\mathcal{N}}(I);\; \|f\|_{\alpha, \mathcal{N}}<\infty\} .$$
$\mathcal{B}^{\alpha}_{\mathcal{N}}(I) $  endowed with the norm
$\|\cdot\|_{\alpha, \mathcal{N}}$ is a non separable Banach space. We introduce $\mathcal{B}^{\alpha, 0}_{\mathcal{N}}(I)=\{f\in L_{\mathcal{N}}(I);\;\; \Delta_{\mathcal{N}}(f, I)(t)=o(t^{\alpha})\;\text{as}\;t\rightarrow 0^+\} $ a separable subspace of $\mathcal{B}^{\alpha}_{\mathcal{N}}(I)$.

We will restrict ourselves to the interval $I=[0,1]$,
so we will omit to precise the interval $I$ in our notations, e.g. we will use
$\mathcal{B}^{\alpha}_{\mathcal{N}} $ to denote the Besov-Orlicz space  $ \mathcal{B}^{\alpha}_{\mathcal{N}}(|0,1])$. With the same notations as in Theorem \ref{Besov}, we have the following isomorphism theorem (see Ciesielski \cite{Ciesielski} or Ciesielski \textit{et al.} \cite{Roynette}):
\begin{thm}\label{Besov Orlicz}
We have
\begin{enumerate}
  \item  $\mathcal{B}^{\alpha}_{\mathcal{N}}$ is linearly isomorphic to a sequences space and we have the following equivalence of norms:
      $$||f||_{\alpha, \mathcal{N}}\sim\max\left\{|f_0|,|f_1|,\sup_{p,j}\frac{1}{\sqrt{p}}2^{-j\left(\tfrac{1}{2}-\alpha+\tfrac{1}{p}\right)}\left[\sum_{k=1}^{2^{j}}|f_{jk}|^p\right]^{\tfrac{1}{p}} \right\},$$
  \item $f$ belongs to $\mathcal{B}^{\alpha, 0}_{ \mathcal{N}} $ if and only if
  $$\lim_{j\to \infty}\sup_{p\geq 1}\frac{1}{\sqrt{p}}2^{-j\left(\tfrac{1}{2}-\alpha+\tfrac{1}{p}\right)}\left[\sum_{k=1}^{2^j}|f_{jk}|^p\right]^{\tfrac{1}{p}}=0.$$
\end{enumerate}
\end{thm}
\begin{rem}
   For any  $p\in [1,\infty)$ and $0<\alpha<1$, the following injections are easy to verify
    $$\mathcal{B}^{\alpha}_{\mathcal{N}}\hookrightarrow\mathcal{B}^{\alpha}_{p}\quad\text{and}\quad\mathcal{B}^{\alpha, 0}_{\mathcal{N}}\hookrightarrow\mathcal{B}^{\alpha, 0}_{p}.$$
\end{rem}

\section{Besov regularity of solution to stochastic heat equation}
\subsection{Linear stochastic heat equation}\label{Linear stochastic heat equation}
Consider the linear stochastic heat equation defined by \eqref{EQ0}, where $\frac{\partial^2}{\partial t\partial x}W(t,x)$ is a space-time white noise, i.e. $(W(t,x),\;t\geq 0,\;x\in \R)$ is a centered Gaussian process with covariance function given by
$$\E[W(t,A)W(t',A')]=(t\wedge t')\, \lambda (A\cap A'),\quad A,A'\in \mathcal{B}_b(\R),\;t,t'\geq 0\,,$$
where $\lambda$ is the Lebesgue measure, $\mathcal{B}_b(\R)$ is the collection of bounded Borel sets and $W(t,x)=W(t,[0,x])$. It is well known that there exist a unique (mild) solution
of this equation given by
\begin{equation}\label{mild solution}
  u(t,x)=\int_{0}^{t}\int_{\R}G(t-s,x-y)dW(s,y),
\end{equation}
where the integral is the Wiener integral with respect to the Gaussian process $W$
and $G$ is the Green kernel of the heat equation given by
\begin{equation}\label{G}
  G(t,x)=\begin{cases}(2\pi t)^{-1/2}\exp(-\tfrac{|x|^2}{2t})\quad & \mbox{if }t>0,\;x\in \R\\
                                                                 0 & \mbox{if }t\leq 0,\;x\in \R .                                                                   \end{cases}
\end{equation}

\subsection{Besov regularity of $t\to u(t,x)$}
Our main result in this paragraph, is the following theorem
\begin{thm}\label{pricipal t}
For all $x\in \R$ and $4<p<\infty$, we have
$$\mathbb{P}(u(.,x)\in \mathcal{B}^{1/4}_{p} )=1\;\;\text{and}\;\;\mathbb{P}(u(.,x)\in \mathcal{B}^{1/4,0}_{p})=0,$$
where $u(.,x)$ is the sample path $t\in [0,1]\longrightarrow u(t,x)$.
\end{thm}
To prove this theorem, we are going to adapt the method used in \cite{Roynette}. For any $x\in \R$, let
$$u_{jk}=2\cdot 2^{j/2}\left\{u\left(\frac{2k-1}{2^{j+1}},x\right)-\frac{1}{2}u\left(\frac{2k}{2^{j+1}},x\right)-\frac{1}{2}u\left(\frac{2k-2}{2^{j+1}},x\right)\right\}.$$
We can rewrite $u_{jk}$ as
\begin{equation}\label{ujk}
\begin{split}
  & u_{jk}\\
&=2\cdot 2^{j/2}\left\{\int_{0}^{\frac{2k-2}{2^{j+1}}}\left[ G\left(\frac{2k-1}{2^{j+1}}-\tau,x-y\right)-\frac{1}{2}G\left(\frac{2k}{2^{j+1}}-\tau,x-y\right)\right.\right.\\
  &\qquad\qquad\qquad\qquad\left.-\frac{1}{2}G\left(\frac{2k-2}{2^{j+1}}-\tau,x-y\right)\right]W(d\tau,dy)+\\
     &+\int_{\frac{2k-2}{2^{j+1}}}^{\frac{2k-1}{2^{j+1}}}\left[G\left(\frac{2k-1}{2^{j+1}}-\tau,x-y\right)-\frac{1}{2}G\left(\frac{2k}{2^{j+1}}-\tau,x-y\right)\right]W(d\tau,dy)\\
     &\left.+ \int_{\frac{2k-1}{2^{j+1}}}^{\frac{2k}{2^{j+1}}}\left[-\frac{1}{2}G\left(\frac{2k}{2^{j+1}}-\tau,x-y\right)\right]W(d\tau,dy)\right\}\\
    &= 2\cdot 2^{j/2}\left\{I_1(j,k)+I_2(j,k)+I_3(j,k)\right\}.\\
\end{split}
\end{equation}
It is important to note that  $I_1(j,k),I_2(j,k)$ and $ I_3(j,k)$ are independent terms. We also put
\begin{equation}\label{vjk}
  v_{jk}=\frac{u_{jk}}{\sigma_{jk}}\;\;\;\text{with}\;\;\; \sigma_{jk}=\left\{\E[|u_{jk}|^2]\right\}^{1/2}.
\end{equation}
 We first state some preliminary results.
\begin{lem}\label{estimation GG}
  For all $x\in \R$, $t,s\in [0,1]$ and $0\leq r_1<r_2\leq t\wedge s$, we have
  $$\int_{r_1}^{r_2}\int_{\R}G(t-\tau,x-y)G(s-\tau,x-y)dy d\tau=\frac{1}{\sqrt{2\pi}}\left(\sqrt{t+s-2r_1}-\sqrt{t+s-2r_2}\right).$$
\end{lem}
\begin{proof} The proof is a consequence of
 successive elementary changes of variables.
\end{proof}

The following lemma is a useful tool to get precise estimations in our calculations. For the proof, we refer to \cite{Roynette}.
\begin{lem}\label{cauchy gaussian}
  Let $(X,Y)$ be a mean zero Gaussian vector such that $\E(X^2)=\E(Y^2)=1$ and $\rho=|\E XY|$. Then for any measurable functions $f$ and $g$ such that
  $\E(f(X))^2<\infty,\E(f(Y))^2<\infty$ and $f(X)$, $f(Y)$ are centred, we have
  $$|\E f(X)g(Y)|\leq \rho \left\{\E(f(X))^2\right\}^{1/2}\left\{\E(g(Y))^2\right\}^{1/2},$$
  when $f$ (or $g$) is even, we can replace $\rho$ by $\rho^2$ in the previous  inequality.
\end{lem}
By using the equality  \eqref{ujk} and Lemma \ref{estimation GG}, we have for all $j\geq 1$ and $k,k'\in \{1,...,2^j\}$ such that $k\neq k'$
\begin{equation}\label{ujkujkprime}
  \E[u_{jk}u_{jk'}]=\frac{2\cdot 2^{j/2}}{\sqrt{\pi}}(\Delta^4\Phi_{k,k'}(0)-\Delta^4\Psi_{k,k'}(0)),
\end{equation}
where  $\Delta^4\Phi_{k,k'},\Delta^4\Psi_{k,k'}$  are the one step progressive differences of order $4$ of the functions
$$\Phi_{k,k'}(x)=(2(k+k')-x)^{1/2}\;\;\text{and}\;\;\Psi_{k,k'}(x)=(2|k-k'|-2+x)^{1/2}.$$
We also have, for all  $j\geq 1$ and $k\in \{1,...,2^j\}$
\begin{equation}\label{Eujk2}
  \E[|u_{jk}|^2]=\frac{2\cdot 2^{j/2}}{\sqrt{\pi}}(\Delta^4\Phi_{k,k}(0)+2-\frac{1}{\sqrt{2}}).
\end{equation}
\begin{lem}\label{estiimation}
  For all $j\geq 1$ and $k,k'\in \{1,...,2^j\}$ with $k\neq k'$, there exists a constant $c_{k,k'}\in (0,1)$ such that
  \begin{equation}\label{estmation Eujkujkprime}
    |\E[u_{jk}u_{jk'}]|\leq \frac{15}{8\sqrt{\pi}}\frac{2^{j/2}}{(2|k-k'|-2+c_{k,k'})^{7/2}}.
  \end{equation}
  And there exists a constant $m>0$ such that, for all $j\geq 1$ and $k\in \{1,...,2^j\}$,
  \begin{equation}\label{estmation Eujk2}
    m2^{j/2} \leq\E[|u_{jk}|^2] \leq \frac{4-\sqrt{2}}{\sqrt{\pi}} 2^{j/2}.
  \end{equation}
\end{lem}
\begin{proof}
  Denote by $\Phi_{k,k'}^{(4)},\Psi_{k,k'}^{(4)}$ the derivatives of order $4$ of respectively $\Phi_{k,k'}$ and $\Psi_{k,k'}$. By successive mean value theorem and \eqref{ujkujkprime}, there exist two constants $0<\beta_{k,k'}<4$ and $0<\gamma_{k,k'}<4$ such that
  \begin{align*}
    \E[u_{jk}u_{jk'}] &= \frac{2\cdot 2^{j/2}}{\sqrt{\pi}}(\Phi_{k,k'}^{(4)}(\beta_{k,k'})-\Psi_{k,k'}^{(4)}(\gamma_{k,k'})) \\
                      &= \frac{15.2^{j/2}}{8\sqrt{\pi}}\left(\frac{1}{(2|k-k'|-2+\gamma_{k,k'})^{7/2}}-\frac{1}{(2(k+k')-\beta_{k,k'})^{7/2}}\right)\\
                      & \leq \frac{15}{8\sqrt{\pi}}\frac{2^{j/2}}{(2|k-k'|-2+\gamma_{k,k'}\wedge 1)^{7/2}}\,.
  \end{align*}
  Then we get \eqref{estmation Eujkujkprime} with $c_{k,k'}=\gamma_{k,k'}\wedge 1$.\par
  To show the upper bound in \eqref{estmation Eujk2}, we use the mean value theorem and \eqref{Eujk2}. So, there exists a constant $0<\beta_{k,k}<4$ such that
  \begin{align}
    \E[u_{jk}^2] &= \frac{2\cdot 2^{j/2}}{\sqrt{\pi}}(\Phi_{k,k}^{(4)}(\beta_{k,k})+2-\frac{1}{\sqrt{2}})\\
     &= \frac{2\cdot 2^{j/2}}{\sqrt{\pi}}(-\frac{15}{16}\frac{1}{(4k-\beta_{k,k})^{7/2}}+2-\frac{1}{\sqrt{2}})\label{c3}\\
     &\leq \frac{2}{\sqrt{\pi}}(2-\frac{1}{\sqrt{2}})2^{j/2}.\nonumber
  \end{align}
  On the other hand to prove the lower bound, first we remark that for $k=1$, we have by \eqref{Eujk2}
  $$\E[|u_{j1}|^2]=\frac{2}{\sqrt{\pi}}(\frac{7\sqrt{2}}{4}-\sqrt{3}+\frac{3}{2})2^{j/2},$$
and for $k\in \{2,...,2^j\}$, we get by \eqref{c3}
  $$\frac{2}{\sqrt{\pi}}(-\frac{15}{16}\frac{1}{4^{7/2}}+2-\frac{1}{\sqrt{2}})2^{j/2}\leq\E[|u_{jk}|^2].$$
  So the lower bound  in \eqref{estmation Eujk2} is obtained with
  $$m=\frac{2}{\sqrt{\pi}}(-\frac{15}{16}\frac{1}{4^{7/2}}+2-\frac{1}{\sqrt{2}})\wedge \frac{2}{\sqrt{\pi}}(\frac{7\sqrt{2}}{4}-\sqrt{3}+\frac{3}{2}).$$
  This finishes the proof of Lemma \ref{estiimation}.
\end{proof}
\begin{lem}\label{vjkvjk rimep}
  There exists a constant $M>0$ such that, for all $j\geq 1$ and $k,k'\in \{1,...,2^j\}$, we have
  \begin{equation}\label{sumEujujprime}
    \sum_{k,k'=1}^{2^j}|\E v_{jk}v_{jk'}|^2\leq M 2^j,
  \end{equation}
  where $v_{jk}$ is given by \eqref{vjk}.
\end{lem}
\begin{proof}
Equality \eqref{vjk} gives
  \begin{align}
    \sum_{k,k'=1}^{2^j}|\E v_{jk}v_{jk'}|^2 &= 2\sum_{k'<k \atop k-k'=1}^{2^j}|\E v_{jk}v_{jk'}|^2+2\sum_{k'<k \atop k-k'\geq 2}^{2^j}|\E v_{jk}v_{jk'}|^2 +\sum_{k=1}^{2^j}\{\E [|v_{jk}|^2]\}^2\nonumber\\
                                            &=2\sum_{k'<k \atop k-k'=1}^{2^j}\left|\frac{\E u_{jk}u_{jk'}}{\sigma_{jk}\sigma_{jk'}}\right|^2+2\sum_{k'<k \atop k-k'\geq 2}^{2^j}\left|\frac{\E u_{jk}u_{jk'}}{\sigma_{jk}\sigma_{jk'}}\right|^2 +2^j \label{sumkprime k}\\
                                            &=2I_1+2I_2+2^j.\nonumber
  \end{align}
  First we are going to estimate $I_1$. Since $k'=k-1$, we have by \eqref{ujkujkprime}
  \begin{equation}\label{ujkujkprime 2}
    \E[u_{jk}u_{jk'}]=\frac{2\cdot 2^{j/2}}{\sqrt{\pi}}(\Delta^4\Phi_{k,k'}(0)+2(1+2\sqrt{3}-3\sqrt{2})),
  \end{equation}
  and
  \begin{equation}\label{delta 4 phi 1}
    \Delta^4\Phi_{k,k'}(0)=\Phi_{k,k'}^{(4)}(\beta_{k,k'})=-\frac{15}{16}\frac{1}{(4k-2-\beta_{k,k'})^{7/2}},
  \end{equation}
  where $\beta_{k,k'}\in (0,4)$.
  Combining \eqref{ujkujkprime 2} and \eqref{delta 4 phi 1}, we obtain
  \begin{equation}\label{delta 4 phi 1 and 2}
    \E[u_{jk}u_{jk'}]=\frac{2\cdot 2^{j/2}}{\sqrt{\pi}}\left(-\frac{15}{16}\frac{1}{(4k-2-\beta_{k,k'})^{7/2}}+2(1+2\sqrt{3}-3\sqrt{2})\right).
  \end{equation}
  So, we get
  \begin{equation}\label{positive estimate}
    \E[u_{jk}u_{jk'}]\leq \frac{4\cdot 2^{j/2}}{\sqrt{\pi}}(1+2\sqrt{3}-3\sqrt{2}),
  \end{equation}
  and
  \begin{equation}\label{negative estimate}
  \begin{split}
     -\E[u_{jk}u_{jk'}] &\leq \frac{15\cdot 2^{j/2}}{8\sqrt{\pi}}\frac{1}{(4k-2-\beta_{k,k'})^{7/2}} \\
       &\leq \frac{15\cdot 2^{j/2}}{8\sqrt{\pi}}\frac{1}{2^{7/2}}.
  \end{split}
  \end{equation}
  Combining  \eqref{positive estimate} and \eqref{negative estimate}, we obtain that
  \begin{equation}\label{the estimate}
    \left|\E[u_{jk}u_{jk'}]\right|\leq K2^{j/2},
  \end{equation}
  where
  $K=\frac{4(1+2\sqrt{3}-3\sqrt{2})}{\sqrt{\pi}}\vee \frac{15}{2^{13/2}\sqrt{\pi}}.$
  So by \eqref{estmation Eujk2} and \eqref{the estimate}, we have
  $$I_1\leq \left(\frac{K}{m}\right)^2(2^j-1).$$
  Now we are going to estimate $I_2$. If we note $C=\left(\frac{15}{m8\sqrt{\pi}}\right)^2$, we have by \eqref{estmation Eujkujkprime} and\eqref{estmation Eujk2},
  \begin{align*}
    I_2 \leq C\,\sum_{k'<k \atop k-k'\geq 2}^{2^j}\frac{1}{(2(k-k')-2+c_{k,k'})^{7}}\\
     &= C\,\sum_{k=3}^{2^j}\sum_{k'=1}^{k-2}\frac{1}{(2(k-k')-2+c_{k,k'})^{7}} \\
     &\leq C\,\sum_{k=3}^{2^j}\sum_{k'=1}^{k-2}\frac{1}{(2(k-k')-2)^{7}} \\
     &\leq C\, \sum_{k=3}^{2^j}\left(\sum_{k'=1}^{k-2}\int_{2(k-k')-3}^{2(k-k')-2}\frac{1}{x^7}dx\right)\\
     &\leq C\, \sum_{k=3}^{2^j}\left(\int_{1}^{2k-4}\frac{1}{x^7}dx\right)
     \leq \frac{C}{6}\, (2^j-2).
  \end{align*}
  This finishes the proof of the lemma  \ref{vjkvjk rimep}.
\end{proof}
\begin{lem}\label{vjk mmoins cp}
For all $j\geq1$ and $k\in \{1,...,2^j\}$, we have
  \begin{equation}\label{vjk moins cp}
    \E\left[\sum_{k=1}^{2^j}(|v_{jk}|^p-c_p)\right]^2\leq (c_{2p}-c_p^2)M 2^j,
  \end{equation}
  where $
    c_p=\frac{1}{\sqrt{2\pi}}\int_{\R}|x|^pe^{-\tfrac{x^2}{2}}dx. $
\end{lem}
\begin{proof}
  First we have
  $$\E\left[\sum_{k=1}^{2^j}(|v_{jk}|^p-c_p)\right]^2=\sum_{k,k'=1}^{2^j}\E\left[(|v_{jk}|^p-c_p)(|v_{jk'}|^p-c_p)\right].$$
  And by applying Lemma \ref{cauchy gaussian}, with $f(x)=g(x)=|x|^p-c_p$,  we get
  $$\E\left[\sum_{k=1}^{2^j}(|v_{jk}|^p-c_p)\right]^2\leq (c_{2p}-c_p^2)\sum_{k,k'=1}^{2^j}|\E[v_{jk}v_{jk'}]|^2.$$
 Inequality (\ref{vjk moins cp}) of Lemma \ref{vjkvjk rimep} ends the proof of the lemma \ref{vjk mmoins cp}.
\end{proof}
Now, we are ready to finish the proof of Theorem \ref{pricipal t}
\begin{proof}[Proof of Theorem \ref{pricipal t}]
We are going to show that, almost surly
\begin{equation}\label{convergence to}
  2^{-j}\sum_{k=1}^{2^j}|v_{jk}|^p\underset{j \to \infty}\longrightarrow c_p.
\end{equation}
For this end we will prove that for all $\varepsilon>0$, we have
\begin{equation}\label{serie}
  \sum_{j\geq 1} \mathbb{P}\left\{2^{-j}\sum_{k=1}^{2^j}|v_{jk}|^p\notin [c_p-\varepsilon,c_p+\varepsilon]\right\}<\infty.
\end{equation}
Markov's inequality ensures
\begin{equation}\label{inferieur}
  \mathbb{P}\left\{2^{-j}\sum_{k=1}^{2^j}|v_{jk}|^p\notin [c_p-\varepsilon,c_p+\varepsilon]\right\}\leq \frac{1}{\varepsilon^22^{2j}}\E\left[\sum_{k=1}^{2^j}(|v_{jk}|^p-c_p)\right]^2.
\end{equation}
Combining the inequality \eqref{inferieur} and Lemma  \ref{vjk mmoins cp}, we get that \eqref{serie} holds and \eqref{convergence to} is then a consequence of   Borel-Cantelli Lemma. Finally, our main result, Theorem \ref{pricipal t}, is a simple consequence of Theorem \ref{Besov}.
\end{proof}

Below we conclude a stronger regularity result than Theorem \ref{pricipal t}
\begin{thm}\label{pricipal Besov Orlicz t}
Let ${\mathcal N}(x)= e^{x^2}-1 $. For all $x\in \R$, we have
$$\mathbb{P}(u(.,x)\in \mathcal{B}^{1/4}_{{\mathcal N}} )=1\;\;\text{and}\;\;\mathbb{P}(u(.,x)\in \mathcal{B}^{1/4,0}_{{\mathcal N}})=0,$$
where $u(.,x)$ is the sample path $t\in [0,1]\longrightarrow u(t,x)$.
\end{thm}
\begin{proof}
 The proof is similar to that one of \cite[Theorem II.5]{Roynette}. Indeed, we use Lemma \ref{vjk mmoins cp} and the fact that for  positive integer $p$, we have
 $$c_{2p}=(2p)!/(p!2^p)\underset{p\to \infty}{\thicksim}e^{-p}2^{p+1/2}p^p.$$
 Therefore, there exists a constant $c$ with $c>1$, such that $c_{2p}\leq c e^{-p}(2p)^p.$
\end{proof}
\begin{rem}
 Lemma \ref{estimation GG} shows clearly that the process $(u(t, x), t\geq0)$ is, up to a constant, a bifractional Brownian motion with parameters $H=K=\frac{1}{2}$. We believe that, with the same arguments, we can prove that $\mathbb{P}[ W^{H,K}\in \mathcal{B}^{HK}_{p}]=1 $ and $\mathbb{P}[ W^{H,K}\in \mathcal{B}^{HK, 0}_{p}]=0 $, where $(W^{H,K}_{t},\,\, t\in[0,1])$ is a bifractional Brownian motion with parameters $H\in (0,1)$ and $K\in (0,1]$. We investigate this result in an upcoming paper.
\end{rem}
\subsection{Besov regularity of $x\to u(t, x)$}
In this paragraph we will study the Besov regularity of $(u(t,x),\;x\in [0,1])$ for any fixed $t\in (0,T]$.
\begin{thm}\label{pricipal x}
For all $t\in (0,T]$ and $2<p<\infty$, we have
$$\mathbb{P}[u(t,.)\in \mathcal{B}^{1/2}_{p}]=1\;\;\text{and}\;\;\mathbb{P}[u(t,.)\in \mathcal{B}^{1/2,0}_{p}]=0,$$
where $u(t,.)$ is the sample path $x\in [0,1]\longrightarrow u(t,x)$.
\end{thm}
As before we first state some preliminary lemmas
\begin{lem}\label{estimation GG x}
  For all $t\in (0,T]$, $x,y\in [0,1]$, we have
  $$\int_{0}^{t}\int_{\R}G(t-\tau,x-\xi)G(t-\tau,y-\xi)d\xi d\tau=F(x-y),$$
  where
  \begin{equation}\label{F}
    F(u)=\int_{0}^{t}\frac{1}{2\sqrt{\pi r}}e^{-\tfrac{u^2}{4r}}dr.
  \end{equation}
\end{lem}
\begin{proof}
  Straightforward calculations.
\end{proof}
\medskip\par\noindent
Define for any $t\in (0,T]$,
\begin{equation}\label{zjk}
  z_{jk}=2\cdot 2^{j/2}\left\{u\left(t,\frac{2k-1}{2^{j+1}}\right)-\frac{1}{2}u\left(t,\frac{2k}{2^{j+1}}\right)-\frac{1}{2}u\left(t,\frac{2k-2}{2^{j+1}}\right)\right\}\,,
\end{equation}
and set
\begin{equation}\label{wjk}
  \omega_{jk}=\frac{z_{jk}}{\varrho_{jk}}\;\;\;\text{with}\;\;\;\varrho_{jk}=\{E[|z_{jk}|^2]\}^{1/2}.
\end{equation}
By using the equality \eqref{zjk} and Lemma \ref{estimation GG x} we obtain, for all $j\geq 1$ and $k,k'\in \{1,...,2^j\}$,
\begin{equation}\label{zjkzjk prime}
  \E[z_{jk}\,z_{jk'}]=2^j\Delta^4\Lambda_{j,k,k'}(0),
\end{equation}
where $\Delta^4\Lambda_{j,k,k'}$ is the one step progressive difference of order $4$ of the function
$$\Lambda_{j,k,k'}(x)=F\left(\frac{2(k-k')-2+x}{2^{j+1}}\right),$$
where $F$ is given by \eqref{F}.

\begin{lem}\label{lem estimation zjkzjkprime}
  For all $j\geq 1$ and $k,k'\in \{1,...,2^j\}$ such that $k\neq k'$, there exists a constant $\kappa_{j,k,k'}\in (0,4)$, such that
  \begin{equation}\label{estimate Ezjk zjk prime}
    |\E[z_{jk}\,z_{jk'}]|\leq \frac{3\theta}{16\sqrt{\pi}}\frac{1}{|2(k-k')-2+\kappa_{j,k,k'}|^3},
  \end{equation}
  where $\theta=\int_{0}^{\infty}\frac{e^{-\frac{1}{4s}}}{s^{5/2}}\left[1+\frac{1}{s}+\frac{1}{12}\frac{1}{s^2}\right]ds.$
  And there exist $m_{1},m_{2}>0$  such that, for all  $j\geq 1$ and $k\in \{1,...,2^j\}$,
  \begin{equation}\label{estimate Ezjk}
    m_{1}\leq \E[|z_{jk}|^2]\leq m_{2}.
  \end{equation}
\end{lem}
\begin{proof}
First recall that by \eqref{zjkzjk prime}, we have
\begin{equation}\label{zjkzjk prime prime}
  \E[z_{jk}\,z_{jk'}]=2^j\Delta^4\Lambda_{j,k,k'}(0).
\end{equation}
Denote by $\Lambda_{j,k,k'}^{(4)}$ the derivative of order 4 of $\Lambda_{j,k,k'}.$  By the mean value theorem and  \eqref{zjkzjk prime prime}, for all $j\geq 1$ and $k,k'\in \{1,...,2^j\}$ such that $k\neq k'$, there exists a constant $\kappa_{j,k,k'}\in (0,4)$ such that
\begin{equation}\label{deltaz}
    \E[z_{jk}\,z_{jk'}]=2^j\Lambda_{j,k,k'}^{(4)}(\kappa_{j,k,k'}).
  \end{equation}
 Taking into account the fact that $k\neq k'$, elementary calculation of  $\Lambda_{j,k,k'}^{(4)}$ entails
  \begin{equation}\label{E zjkzjkprime}
  \begin{split}
      \E[z_{jk}\,z_{jk'}]=\frac{3}{16\sqrt{\pi}}&\frac{1}{|2(k-k')-2+\kappa_{j,k,k'}|^3}\times  \\
       & \int_{0}^{\nu}\frac{e^{-\frac{1}{4s}}}{s^{5/2}}\left[1-\frac{1}{s}+\frac{1}{12}\frac{1}{s^2}\right]ds,
  \end{split}
  \end{equation}
where $\nu=t\left(\frac{2^{j+1}}{2(k-k')-2+\kappa_{j,k,k'}}\right)^2$. Then we get \eqref{estimate Ezjk zjk prime} by taking
$\theta=\int_{0}^{\infty}\frac{e^{-\frac{1}{4s}}}{s^{5/2}}\left[1+\frac{1}{s}+\frac{1}{12}\frac{1}{s^2}\right]ds.$
 \par

Now we are going to show \eqref{estimate Ezjk}. First, we start by proving the upper bound. From \eqref{zjkzjk prime} and \eqref{deltaz} we get that for all $j \geq 1$ and $k\in \{1,...,2^j\}$,
  \begin{align*}
    \E[|z_{jk}|^2] =2^j\Delta^4\Lambda_{j,k,k}(0)=2^j \int_{0}^{t}\frac{1}{2\sqrt{\pi r}}\Delta^4\psi_{j,r}(0)dr,
  \end{align*}
  where
$$\psi_{j,r}(x)=\exp\left(-\frac{1}{4r}\left(\frac{x-2}{2^{j+1}}\right)^2\right).$$
   By the change of variable $s=2^{2(j+1)}r$, we have that
   \begin{equation}\label{zjk phi s}
   \E[|z_{jk}|^2]=\int_{0}^{t2^{2(j+1)}}\frac{1}{2\sqrt{\pi s}}\Delta^4\varphi_{s}(0)ds,
   \end{equation}
   where $\varphi_{s}(x)=\frac{e^{-\frac{1}{4s}(x-2)^2}}{2}.$
   Remark that $\varphi_{s}$ depends neither  on $j$  nor on $k$. Again by mean value theorem there exists a constant $\lambda_s\in (0,4)$ depending only on $s$, such that
   \begin{equation}\label{varphi4}
     \Delta^4\varphi_{s}(0)=\varphi^{(4)}_{s}(\lambda_s)=\frac{3}{8}\frac{e^{-\frac{1}{4s}(\lambda_s-2)^2}}{s^2}\left\{1-\frac{(\lambda_s-2)^2}{s}+\frac{1}{12}\frac{(\lambda_s-2)^4}{s^2}\right\}.
   \end{equation}
   We use the formula \eqref{varphi4} away from the origin together with  \eqref{zjk phi s}, to get
   \begin{equation}\label{zjk leq}
    \begin{split}
     \E[|z_{jk}|^2] \leq  &\int_{0}^{4t}\frac{1}{2\sqrt{\pi s}}\Delta^4\varphi_{s}(0)ds\\
                          &+\int_{4t}^{\infty}\frac{3}{8}\frac{e^{-\frac{1}{4s}(\lambda_s-2)^2}}{s^2}\left\{1-\frac{(\lambda_s-2)^2}{s}+\frac{1}{12}\frac{(\lambda_s-2)^4}{s^2}\right\}ds.
     \end{split}
   \end{equation}
   Then \eqref{zjk leq} entails
   \begin{equation}\label{zjk leq 2}
     \E[|z_{jk}|^2] \leq  \int_{0}^{4t}\frac{2}{\sqrt{\pi s}}ds+\int_{4t}^{\infty}\frac{3}{8}\frac{1}{s^2}\left\{1+\frac{4}{s}+\frac{4}{3}\frac{1}{s^2}\right\}ds\,:=\, m_{2}\,.
   \end{equation}
   For the lower bound of \eqref{estimate Ezjk}, we have by  \eqref{zjk phi s}
   \begin{equation}\label{zjk phi s minoration}
    \begin{split}
   &\E[|z_{jk}|^2]\\
   &=\frac{1}{2\sqrt{\pi}}\left\{\int_{0}^{t2^{2(j+1)}}\frac{1}{\sqrt{s}}e^{-\frac{1}{s}}ds-4\int_{0}^{t2^{2(j+1)}}\frac{1}{\sqrt{s}}e^{-\frac{1}{4s}}ds+3\int_{0}^{t2^{2(j+1)}}\frac{1}{\sqrt{s}}ds\right\}\\
   &\geq \frac{1}{2\sqrt{\pi}}\int_{0}^{4t}\frac{1}{\sqrt{s}}\left(e^{-\frac{1}{s}}-4e^{-\frac{1}{4s}}+3\right)ds \,:= m_{1},
   \end{split}
   \end{equation}
the last inequality holds because the function $x\in (0,\infty)\to
  \int_{0}^{x}\frac{1}{\sqrt{s}}e^{-\frac{1}{s}}ds-4\int_{0}^{x}\frac{1}{\sqrt{s}}e^{-\frac{1}{4s}}ds+3\int_{0}^{x}\frac{1}{\sqrt{s}}ds$ is nondecreasing. So we obtain
a lower bound $m_1$, which finishes the proof of Lemma \ref{lem estimation zjkzjkprime}.
\end{proof}
\begin{lem}
  There exists a constant $K>0$ such that, for all $j \geq1$ and $k,k'\in \{1,...,2^j\}$, we have
  \begin{equation}\label{wjkwjk prime}
    \sum_{k,k'=1}^{2^j}|\E\,\omega_{jk}\,\omega_{jk'}|^2\leq K 2^j,
  \end{equation}
  where $\omega_{jk}$ is given by \eqref{wjk}.
\end{lem}
\begin{proof}
  The proof uses the same calculations as those used in Lemma \ref{vjkvjk rimep}.
\end{proof}
We utilize similar arguments as in the proof of Lemma  \ref{vjk mmoins cp} to obtain
\begin{lem}
For all $j\geq1$ and $k\in \{1,...,2^j\}$, we have
  \begin{equation}
    \E\left[\sum_{k=1}^{2^j}(|\omega_{jk}|^p-c_p)\right]^2\leq (c_{2p}-c_p^2)K 2^j,
  \end{equation}
  where $
    c_p=\frac{1}{\sqrt{2\pi}}\int_{\R}|x|^pe^{-\tfrac{x^2}{2}}dx.$
\end{lem}
\begin{proof}[Proof of Theorem \ref{pricipal x}]
We deduce the proof  by the same arguments as those used in Theorem \ref{pricipal t}.
\end{proof}
As above one can also show the following
\begin{thm}\label{pricipal Besov Orlicz x}
Let ${\mathcal N}(x)= e^{x^2}-1 $. For all $t\in (0,T]$, we have
$$\mathbb{P}(u(t, .)\in \mathcal{B}^{1/2}_{{\mathcal N}} )=1\;\;\text{and}\;\;\mathbb{P}(u(t, .)\in \mathcal{B}^{1/2, 0}_{{\mathcal N}})=0,$$
where $u(t, .)$ is the sample path $x\in [0,1]\longrightarrow u(t,x)$.
\end{thm}
\begin{rem}
One can easily deduce a sharp Besov regularity, with respect to space variable, of the following stochastic heat equation with a non-null initial condition:
\begin{equation}\label{EQ1}
    \begin{split}
       \frac{\partial v}{\partial t}(t,x) &= \frac{1}{2}\Delta v(t,x)+\frac{\partial^2}{\partial t\partial x}W(t,x),\quad t>0,\;x\in \R, \\
         v(0,x) &= g(x),\qquad x\in \R,
    \end{split}
\end{equation}
where $g$ is a bounded measurable function. The (mild) solution of (\ref{EQ1}) is given by,
\begin{align}\label{mild solution1}
  v(t,x) &= \int_{\R}G(t,x-y) g(y)\, dy+\int_{0}^{t}\int_{\R}G(t-s,x-y)dW(s,y)\\
         &:= J(t,x)+u(t,x).\nonumber
\end{align}
\begin{itemize}
\item Suppose that  $g\in \mathcal{B}^{\alpha}_{p_0}(\R)$ for some $ 1< p_0<\infty$ and $0<\alpha<1$.
Then one can easily verify that,  for any fixed $t\in (0,T]$, the function  $J(t,\cdot):x\in [0,1]\mapsto J(t,x)$ belongs to $ \mathcal{B}^{\alpha}_{p_0}$. Consequently, standard continuous injections imply that, for any $ 2< p\leq p_0 $,
$\mathbb{P}\left[v(t,.)\in \mathcal{B}^{\frac{1}{2}\wedge \alpha}_{p}\right]=1$ .
\item In addition, if $g\in \mathcal{B}^{\alpha}_{\mathcal N}(\R)$, then
$\mathbb{P}\left[v(t,.)\in \mathcal{B}^{\frac{1}{2}\wedge \alpha}_{\mathcal N}\right]=1 $.
\end{itemize}
\end{rem}
\section{Existence and regularity of local times}
\subsection{Local time of the process $t\to u(t,x)$}
In this section we will consider the process $(u(t,x), \,\, t\in [0,T])$, for some fixed $T>0$ and $x\in\R$.
\subsubsection{Existence of the local time}
We will need the following estimates, which can be easily shown by standard arguments as change of variables and Parseval's equality:
\begin{lem}\label{lemma 1}
For any $x\in \R$ and $s<t$ with $s,t\in [0,T]$,
  \begin{enumerate}
    \item $\int_{0}^{s}\int_{\R}|G(t-r,x-y)-G(s-r,x-y)|^2dy dr\leq C|t-s|^{1/2}$
    \item $\int_{s}^{t}\int_{\R}G^2(t-r,x-y)dy dr=C_1|t-s|^{1/2}$.\label{2}
  \end{enumerate}
\end{lem}
\begin{prop}\label{prop 2}
  For any $x\in \R$ and $0\leq p<3$, we have
  $$\int_{0}^{T}\int_{0}^{T}[\E(u(t,x)-u(s,x))^2]^{-(p+1)/2}dsdt<\infty\,.$$
\end{prop}
\begin{proof}
  We have for $s,t\in [0,T]$, such that $s<t$,
  \begin{align*}
    &\E(u(t,x)-u(s,x))^2 \\
    &= \E\left|\int_{0}^{s}\int_{\R}(G(t-\tau,x-y)-G(s-\tau,x-y))dW(\tau,y)+\int_{s}^{t}\int_{\R}G(t-\tau,x-y)dW(\tau,y)\right|^2 \,.
\end{align*}
Independence argument together with the point \eqref{2} of Lemma \ref{lemma 1} give
  \begin{align*}
     &\E(u(t,x)-u(s,x))^2 \\
     & = \int_{0}^{s}\int_{\R}(G(t-\tau,x-y)-G(s-\tau,x-y))^2dy d\tau+\int_{s}^{t}\int_{\R}G^2(t-\tau,x-y)dy d\tau\\
     & \geq \int_{s}^{t}\int_{\R}G^2(t-\tau,x-y)dy d\tau\geq C_1\sqrt{(t-s)}\,,
  \end{align*}
where $C_1$ is a positive constant. This ends the proof of Proposition \ref{prop 2}.
\end{proof}
\begin{thm}\label{thm 1}
  Let $x\in \R$ be fixed, then
  \begin{enumerate}
  \item There exists a square integrable version of the local time of $(u(t,x),\;t\in [0,T])$.
  \item  The process $(u(t,x),\;t\in [0,T])$ satisfies the LND property i.e., formula \eqref{LND}.
\end{enumerate}
\end{thm}
\begin{proof}
    The existence of a square integrable version of the local time of $u(t,x)$ is a consequence of Berman's
theory (cf. Theorem \ref{thmberman}) and Proposition \ref{prop 2}. We will
denote this version by $(L(\xi,t),\; t \geq 0,\; \xi\in \R)$.

Let us now prove that $(u(t,x),\;t\in [0,T])$ satisfies the LND condition. So we have to prove,
\begin{equation}\label{LND u(t,x)}
  \lim_{c\to 0}\inf_{0\leq t-r\leq c,\;r<s<t}\frac{\var(u(t,x)-u(s,x)|u(\tau,x),\;r\leq \tau \leq s)}{\var(u(t,x)-u(s,x))}>0.
\end{equation}
First, remark that we have the following inclusion of $\sigma$-algebras
$$\sigma(u(\tau,x),\;r\leq \tau \leq s)\subset \mathcal{F}^W_s,$$
where we have noted $\mathcal{F}^W_s=\sigma((W(r,y),\;0\leq r\leq s,\;y\in \R)\vee\mathcal{N}).$ We then get
\begin{equation}\label{var var}
  \frac{\var(u(t,x)-u(s,x)|u(\tau,x),\;r\leq \tau \leq s)}{\var(u(t,x)-u(s,x))}\geq \frac{\var(u(t,x)-u(s,x)|\mathcal{F}^W_s)}{\var(u(t,x)-u(s,x))}.
\end{equation}
Now
\begin{equation}\label{var=var}
\begin{split}
   &\var(u(t,x)-u(s,x)|W(\tau,y),\;0\leq \tau \leq s,\;y\in \R)\\
   &= \var(\int_{0}^{s}\int_{\R}(G(t-\tau,x-y)-G(s-\tau,x-y))dW(\tau,y) \\
   &\qquad\qquad\qquad+\int_{s}^{t}\int_{\R}G(t-\tau,x-y)dW(\tau,y)|\mathcal{F}^W_s).
\end{split}
\end{equation}
But since $\int_{0}^{s}\int_{\R}(G(t-\tau,x-y)-G(s-\tau,x-y))dW(\tau,y)$ is $\mathcal{F}^W_s-$measurable and $\int_{s}^{t}\int_{\R}G(t-\tau,x-y)dW(\tau,y)$ is independent of $\mathcal{F}^W_s$, we have
\begin{equation}\label{varCon=}
  \var(u(t,x)-u(s,x)|W(\tau,y),\;0\leq \tau \leq s,\;y\in \R)=\int_{s}^{t}\int_{\R}G^2(t-\tau,x-y)dy d\tau.
\end{equation}
On the other hand
\begin{equation}\label{var=}
\begin{split}
  \var(u(t,x)-u(s,x))=\int_{0}^{s}\int_{\R}&(G(t-\tau,x-y)-G(s-\tau,x-y))^2dy d\tau\\
  &+\int_{s}^{t}\int_{\R}G^2(t-\tau,x-y)dy d\tau\\
  :=A(s,t).
  \end{split}
\end{equation}
Combining \eqref{var var}, \eqref{var=var}, \eqref{varCon=} and \eqref{var=} we obtain
\begin{align*}
   & \lim_{c\to 0}\inf_{0\leq t-r\leq c,\;r<s<t}\frac{\var(u(t,x)-u(s,x)|u(\tau,x),\;r\leq \tau \leq s)}{\var(u(t,x)-u(s,x))}  \\
  & \geq \lim_{c\to 0}\inf_{0\leq t-r\leq c,\;r<s<t} \frac{\int_{s}^{t}\int_{\R}G^2(t-\tau,x-y)dy d\tau}{A(s,t)}.
\end{align*}
Remark that
\begin{align*}
  &\quad\;\lim_{c\to 0}\inf_{0\leq t-r\leq c,\;r<s<t} \frac{\int_{s}^{t}\int_{\R}G^2(t-\tau,x-y)dy d\tau}{A(s,t)}>0\\
  &\Leftrightarrow \lim_{c\to 0}\inf_{0\leq t-r\leq c,\;r<s<t} \frac{\int_{s}^{t}\int_{\R}G^2(t-\tau,x-y)dy d\tau}{\int_{0}^{s}\int_{\R}(G(t-\tau,x-y)-G(s-\tau,x-y))^2 dy d\tau}>0.
\end{align*}
The last property of $G$ is assured by Lemma \ref{lemma 1}, which ends the proof of Theorem \ref{thm 1}.
\end{proof}
\begin{prop}\label{prop 4}
  For all $x,y\in \R$, $t,t+h\in (0,T]$ and for any even positive integer $n$, there exists $C_n>0$ such that
  \begin{align*}
    \E[L(y,t+h)-L(x,t+h)-L(y,t)+L(x,t)]^n &\leq C_n|x-y|^{\zeta n}|h|^{n(3/4-\zeta/4)}, \\
    \E[L(x,t+h)-L(x,t)]^n &\leq C_n |h|^{3n/4},
  \end{align*}
  where $0<\zeta<1$.
\end{prop}
\begin{proof}
  We prove just the first inequality; the second one follows the same lines. For simplicity of notations we use $X_t$ to denote the process $(u(t,x)\,,\,\, t\in[0,T])$. We consider only $h>0$ such that $t+h\in [0,T]$ the other case follows the same way.  Following \cite{GemanHoroviz} or \cite{Dozzi} we have
\begin{align*}
  & \E[L(y,t+h)-L(x,t+h)-L(y,t)+L(x,t)]^n \\
  & =(2\pi)^{-n}\int_{[t,t+h]^n}\int_{\R^{n}}\prod_{j=1}^{n}[e^{-iyu_j}-e^{-ixu_j}]\E[e^{i\sum_{j=1}^{n}u_jX_{t_j}}]d\overline{u}d\overline{t}.
\end{align*}
The elementary inequality $|1-e^{i\theta}|\leq 2^{1-\zeta}|\theta|^{\zeta}$ for any $0<\zeta<1$ and $\theta\in \R$, leads to
\begin{equation}\label{E J()}
  \E[L(y,t+h)-L(x,t+h)-L(y,t)+L(x,t)]^n\leq 2^{-n\zeta}\pi^{-n}|y-x|^{n\zeta}\, T(n,\zeta),
\end{equation}
where
$$T(n,\zeta)=\int_{[t,t+h]^n}\int_{\R^{n}}\prod_{j=1}^{n}|u_j|^{\zeta}\E[e^{i\sum_{j=1}^{n}u_j X_{t_j}}]d\overline{u}d\overline{t}.$$
In order to apply the LND property for the Gaussian process $X_t$, we do two transformations:
\begin{itemize}
  \item We replace the integration over the domain $[t,t+h]^n$ by the integration over the subset $t<t_1<t_2...<t_n<t+h$.
  \item In the integral over the u's, we change the variable of integration by the following transformation
        $$u_n=v_n,\quad u_j=v_j-v_{j+1},\quad j=1,...,n-1.$$
\end{itemize}
We obtain
\begin{align}\label{J() m!}
  &T(n,\zeta)=n!\int_{t<t_1<t_2...<t_n<t+h}\int_{\R^n}\prod_{j=1}^{n-1}|v_j-v_{j+1}|^{\zeta}|v_n|^{\zeta}\E[e^{i\sum_{j=1}^{n}v_j(X_{t_j}-X_{t_{j-1}})}]d\overline{v}d\overline{t}\nonumber\\
 &= n!\int_{t<t_1<t_2...<t_n<t+h}\int_{\R^n}\prod_{j=1}^{n-1}|v_j-v_{j+1}|^{\zeta}|v_n|^{\zeta}\, e^{-\frac{1}{2}var(\sum_{j=1}^{n}v_j(X_{t_j}-X_{t_{j-1}}))}\, d\overline{v}d\overline{t},
\end{align}
where $t_0=0.$
Now, since $|a-b|^{\zeta}\leq |a|^{\zeta}+|b|^{\zeta}$ for all $0<\zeta<1$, it follows that
\begin{equation}\label{||}
  \prod_{j=1}^{n-1}|v_j-v_{j+1}|^{\zeta}|v_n|^{\zeta}\leq \prod_{j=1}^{n-1}(|v_j|^{\zeta}+|v_{j+1}|^{\zeta})|v_n|^{\zeta}.
\end{equation}
Note that the last term in the right is at most equal to a finite sum of terms each of the form $\prod_{j=1}^{n}|v_j|^{\epsilon_j\zeta}$,
where $\epsilon_j=0,1,$ or $2$ and $\sum_{j=1}^{n}\epsilon_j=n$. Let us write for simplicity $\sigma^2(j) = \E(X_{t_j}- X_{t_{j-1}})^2$. Using \eqref{||} and the LND property of $X_t$, i.e. the second point in Theorem \ref{thm 1}, the term $T(n,\zeta)$ in \eqref{J() m!} is dominated by
the sum over all possible choice of $(\epsilon_1,...,\epsilon_n)\in\{0,1,2\}^n$ of the following terms
\begin{equation}\label{use of LND}
  \int_{t<t_1<t_2...<t_m<t+h}\int_{\R^n}\prod_{j=1}^{n}|v_j|^{\epsilon_j\zeta}\exp\left(-\frac{C_n}{2}\sum_{j=1}^{n}v_j^2\sigma^2(j)\right)d\overline{v}d\overline{t},
\end{equation}
where $C_n$ is a positive constant and $h$ is small enough such that
$0<h<\delta_n$, ($\delta_n$ and $C_n$ are given by the LND property). Now, by the change of variable $x_j=\sigma(j)v_j$,
the term \eqref{use of LND} becomes
\begin{equation}\label{change variable}
  \int_{t<t_1<t_2...<t_m<t+h}\prod_{j=1}^{n}\sigma(j)^{-1-\zeta \epsilon_j}\int_{\R^n}\prod_{j=1}^{n}|x_j|^{\epsilon_j \zeta}\exp\left(-\frac{C_n}{2}\sum_{j=1}^{n}x_j^2\right)d\overline{x}d\overline{t}.
\end{equation}
Using the second point in Lemma \ref{lemma 1}, we have $\sigma^2(j)=\E(X_{t_j}-X_{t_{j-1}})^2 \geq C(t_j-t_{j-1})^{1/2}$, where $C$ is a positive constant. This implies that the integral in \eqref{change variable} is dominated by
\begin{equation}\label{|tj-tj-1|}
  \begin{split}
      \int_{t<t_1<t_2...<t_m<t+h}&\prod_{j=1}^{n}|t_j-t_{j-1}|^{\frac{-1-\zeta\epsilon_j}{4}}\int_{\R^n}\prod_{j=1}^{n}|x_j|^{\epsilon_j \zeta}\exp\left(-\frac{C_n}{2}\sum_{j=1}^{n}x_j^2\right)d\overline{x}d\overline{t}  \\
       &= C(n,\zeta)\int_{t<t_1<t_2...<t_m<t+h}\prod_{j=1}^{n}|t_j-t_{j-1}|^{\frac{-1-\zeta\epsilon_j}{4}}d\overline{t}.
  \end{split}
\end{equation}
Now, return to Equation \eqref{E J()}. Combining \eqref{J() m!}, \eqref{use of LND}
and \eqref{|tj-tj-1|} we obtain
\begin{equation}\label{E C(n,xi)}
  \begin{split}
      & \E[L(y,t+h)-L(x,t+h)-L(y,t)+L(x,t)]^n  \\
       & \leq C(n, \zeta)|y-x|^{n\zeta}\int_{t<t_1<t_2...<t_m<t+h}\prod_{j=1}^{n}|t_j-t_{j-1}|^{\frac{-1-\zeta \epsilon_j}{4}}d\overline{t}.
  \end{split}
\end{equation}
Remark that the integral in the right hand side of \eqref{E C(n,xi)} is finite. Moreover, by using an elementary
calculations, we have for any $n \geq 1$, $h > 0$ and $b_j < 1$
$$\int_{t<t_1<t_2...<t_m<t+h}\prod_{j=1}^{n}|t_j-t_{j-1}|^{-b_j}d\overline{t}=h^{n-\sum_{j=1}^{n}b_j}\frac{\prod_{j=1}^{n}\Gamma(1-b_j)}{\Gamma(1+n-\sum_{j=1}^{n}b_j)}.$$
Finally, taking $b_j=\frac{1+\zeta\epsilon_j}{4}$ , we get
\begin{equation}\label{fin}
  \E[L(y,t+h)-L(x,t+h)-L(y,t)+L(x,t)]^n\leq C(n,\zeta)|y-x|^{n\zeta}h^{n(\frac{3}{4}-\frac{\zeta}{4})}.
\end{equation}
\end{proof}
We can deduce by classical arguments (cf. D. Geman-J. Horowitz \cite[Theorem 26.1]{GemanHoroviz} or Berman \cite[Theorem 8.1.]{Berman2}) the following regularity result for the local time of the solution
$(u(t,x),\,\,\, t\in [0,T])$
\begin{thm}\label{joint continuity}
  For any $x\in \R$, the solution $(u(t,x),\;t\in [0,T])$ has almost surely a jointly continuous local time $(L(\xi,t),\; t\in[0,T],\;\xi\in \R)$ which satisfies for all $\alpha < 3/4$,
\begin{equation}\label{regulalirity L time}
  \sup_{\xi}|L(\xi,t+h)-L(\xi,t)|\leq \eta'h^{\alpha},
\end{equation}
for any $t,t+h\in [0,T]$ such that $|h|<\eta$, where $\eta$ and $\eta'$ are random variables a.s. positive and finite.
\end{thm}
We also obtain by Proposition \ref{prop 4} together with a version of Kolmogorov's continuity theorem in Besov norms (see Boufoussi et al. \cite[Lemma 2.1.]{BLD}), that
\begin{thm}\label{LocalBesov}
  For all $\lambda>0$, $p>\frac{1}{\lambda}$ and $\xi\in\R$,
  $$\mathbb{P}\left(L(\xi,.)\in \mathcal{B}^{\omega_{\lambda}}_p\right)=1,$$
  where $\omega_{\lambda}(t)=t^{3/4}(\log(1/t))^{\lambda}$ and $L(\xi,.)$ is the sample paths $t\to L(\xi,t)$, $t\in[0,1]$ .
\end{thm}
\begin{rem}
\begin{enumerate}
\item Taking $ \lambda $ small enough, Theorem \ref{LocalBesov} ensures a more accurate regularity result. Particularly, we deduce by injection (\ref{Rinject}) that $L(\xi,.)$ satisfies a.s. a $\beta-$H\"{o}lder condition for any $\beta<{\frac{3}{4}}$.
\item The process $ u(. , x) $ satisfies (\ref{(p+1)/2}) of Theorem \ref{thmberman} with $0\leq p<3$. Then there is a version of the local time $L(\xi, t)$, which is differentiable with respect to the space variable, and a.s. $L^{(1)}(\xi, t) = \partial_{\xi} L(\xi, t )\in L^{2}(\R , d\xi) $.\\
It is easy to verify that $L^{(1)}$ satisfies (\ref{E J()}) with $T(n,\zeta+1)$ instead of $T(n, \zeta)$. Following
the same arguments as in Proposition \ref{prop 4}, the finiteness of the integral in (\ref{E C(n,xi)}) (with
$\zeta+1$ in place of $\zeta$) requires that $\zeta < 1/2$. Furthermore, we obtain that for all $x,y\in \R$, $t,t+h \in(0,T]$  and for any positive integer $n$, there exists $C_n>0$ such that
$$\E[L^{(1)}(y,t+h)-L^{(1)}(x,t+h)-L^{(1)}(y,t)+L^{(1)}(x,t)]^n\leq C_n|x-y|^{\zeta n}|h|^{n(1/2-\zeta/4)},$$
where $0<\zeta<1/2.$
\end{enumerate}
\end{rem}
Consequently, we have the following regularity result
\begin{thm}
  There is a jointly continuous version of $(L^{(1)}(\xi,t), t \in [0,T], \xi\in\R)$ satisfying: For all
compact $U \subset\R$ and for any $\alpha < 1/2$
$$\sup_{x,y\in U,x\neq y}\frac{|L^{(1)}(x,t)-L^{(1)}(y,t)|}{|x-y|^{\alpha}}<\infty,\text{ a.s.}$$
\end{thm}
\subsubsection{Hausdorff dimension of level sets}
Let $x\in \R$ be fixed.  We define, for any $\xi\in\R$,
the $\xi$-level set of $(u(t,x)\,,\,\,t\in [0,T])$ by
$$M^x(\xi)=\{t\in [0,T]\;:\;u(t,x)=\xi\}.$$
Our goal is to determine the Hausdorff dimension $\dim_H(M^x(u(t_0,x)))$ of $M^x(u(t_0, x)) $. We can refer to \cite[p. 27]{Falconer} for an introduction to Hausdorff measure and dimension.
One of the crucial applications of the joint continuity of the local time is to extend $L(\xi,.)$ as a finite measure supported on the level set $M^x(\xi)$ see \cite[Theorem 8.6.1]{adler}. To find a lower bound of the Hausdorff dimension of the level sets we need first the following Frostman's Lemma cf. \cite[Lemma 6.10.]{DalangKHOSHNEVISAN}
\begin{lem}\label{Frostman}
  Let $E$ be a Borel set of $\R$. $\mathcal{H}^s(E)>0$ if and only if there exists a finite Borel measure $\mu$ supported on $E$ such that $\mu(E)>0$ and a positive constant $c$ such that
  $$\mu((y-r,y+r))\leq c r^s,$$
  for all $y\in \R$ and $r>0$.
\end{lem}
\begin{lem}\label{lem level set lower}
  For all $x\in \R$, we have almost surely and for almost every $t_0\in[0,T]$
  $$\dim_H(M^x(u(t_0,x)))\geq \frac{3}{4}.$$
\end{lem}
\begin{proof}
 Let $x\in \R$ be fixed, we have by \cite[ Lemma 1.1.]{Berman1} that for almost every $t_0$
  $$L(u(t_0,x),T)>0.$$
 We know that $L(u(t_0,x),.)$ is a measure supported on $M^x(u(t_0,x))$, and  \eqref{regulalirity L time} entails that $L(u(t_0,x),.)$ satisfies a.s. a H\"{o}lder condition of any order smaller than $\tfrac{3}{4}$. So by Lemma \ref{Frostman} we have almost surely and for almost every $t_0$
 $$\dim_H(M^x(u(t_0,x)))\geq \frac{3}{4}.$$
\end{proof}
\begin{lem}\label{lem level set upper}
  For all $x\in \R$,  we have almost surely and for all $t_0\in [0,T]$
  $$\dim_H(M^x(u(t_0,x)))\leq \frac{3}{4}.$$
\end{lem}
\begin{proof}
  We know that $u(.,x)$ satisfies a.s. a H\"{o}lder condition of any order smaller than $\tfrac{1}{4}$. By Theorem \ref{joint continuity} its local time is jointly continuous. The result then follows by  \cite[Theorem 8.7.3.]{adler}.
\end{proof}
Combining Lemma \ref{lem level set lower} and Lemma \ref{lem level set upper} we obtain
\begin{cor}
  For all $x\in \R$, we have almost surely and for almost every $t_0$
  $$\dim_H(M^x(u(t_0,x)))= \frac{3}{4}.$$
\end{cor}
\subsection{Local time of the process $x\to u(t,x)$}
\subsubsection{Existence of the local time}
Let $[a,b]\subset\mathbb{R}$, we will prove the existence of the local time of the process $(u(t,x)\,,\,\,\,x\in [a,b])$ where $t>0$ is fixed.  First we need the following result
\begin{lem}\label{lemma1}
    For fixed $t>0$, and for any $x,y\in [a,b]$, there exists a constant $c_t>0$ such that
    $$c_t|x-y|\leq\E(u(t,x)-u(t,y))^2\leq\frac{|x-y|}{2\pi}.$$
\end{lem}
\begin{proof}
Let $x,y\in [a,b]$ such that $x>y$, the change of variable $r=t-s$ together with Parseval's identity give
\begin{align*}
    \E(u(t,x)-u(t,y))^2 &= \int_0^t\int_{\R}(G(r,x-z)-G(r, y-z))^2 dz dr\\
    &=\frac{1}{2\pi}\int_0^t\int_{\R}\left|e^{ixu}\exp(-\frac{ru^2}{2})-e^{iyu}\exp(-\frac{ru^2}{2})\right|^2du dr\\
                        &=\frac{1}{2\pi}\int_0^t\int_{\R}\exp(-ru^2)\left|e^{i(x-y)u}-1\right|^2 du dr.
\end{align*}
Again by the transformations  $v=u(x-y)$ and $\tau=\frac{r}{(x-y)^2}$, we get
$$\E(u(t,x)-u(t,y))^2=\frac{x-y}{2\pi}\int_0^{\tfrac{t}{(x-y)^2}}\int_{\R}\exp(-\tau v^2)\left|e^{iv}-1\right|^2 dv d\tau.$$
Using Fubini, we obtain
\begin{align*}
\E(u(t,x)-u(t,y))^2 &= \frac{x-y}{2\pi}\int_{\R}\int_0^{\tfrac{t}{(x-y)^2}}\exp(-\tau v^2)\left|e^{iv}-1\right|^2 d\tau dv\\
                    &= \frac{x-y}{2\pi}\int_{\R}(1-\exp(-\frac{t}{(x-y)^2}v^2))\frac{|e^{iv}-1|^2}{v^2}dv\\
                    &= \frac{x-y}{2\pi}\left\{\int_{\R}\frac{|e^{iv}-1|^2}{v^2}dv-\int_{\R}\exp(-\frac{t}{(x-y)^2}v^2)\frac{|e^{iv}-1|^2}{v^2}dv\right\}\\
                    &=\frac{x-y}{2\pi}\left\{1-\int_{\R}\exp(-\frac{t}{(x-y)^2}v^2)\frac{|e^{iv}-1|^2}{v^2}dv\right\}\,,
\end{align*}
where in the last line, Parseval's identity gives
$\int_{\R}\frac{|e^{iv}-1|^2}{v^2}dv=\int_{\R}\chi_{[0,1]}(v)dv=1.$ So, on one hand, it is clear that
$$\E(u(t,x)-u(t,y))^2\leq\frac{|x-y|}{2\pi}.$$
On the other hand,
to find a lower bound for $\E(u(t,x)-u(t,y))^2$, we need to get a constant $0\leq C_t<1$ such that
$$ C_t \geq\int_{\R}\exp(-\lambda v^2)\frac{|e^{iv}-1|^2}{v^2}dv:=A\,,$$
where we have used the notation
$\lambda=\frac{t}{(x-y)^2} .$\\
Now, denote by
$f(v)=\frac{1}{\sqrt{2\pi \lambda}}e^{-v^2/2\lambda} $ and $ g(v)=\chi _{[0,1]}(v).$
It follows by Parseval's identity,
\begin{align*}
    A = \int_{\R}\left|\exp(-\frac{\lambda v^2}{2})\frac{e^{iv}-1}{iv}\right|^2dv
      &= \int_{\R}|\widehat{f*g}(v)|^2dv\\
      &=\int_{\R}|f*g(v)|^2dv .
      \end{align*}
Then
$$A=\int_{\R}\left\{\int_{[0,1]^2}\frac{1}{2\pi \lambda}e^{-(v-z_1)^2/2\lambda}e^{-(v-z_2)^2/2\lambda}dz_1dz_2\right\}dv.$$
By Fubini we have
\begin{align*}
    A &= \int_{[0,1]^2}\frac{1}{\sqrt{2\pi \lambda}}\left\{\int_{\R}\frac{1}{\sqrt{2\pi \lambda}}e^{-(v-z_1)^2/2\lambda}e^{-(v-z_2)^2/2\lambda}dv\right\}dz_1dz_2\\
      &= \int_{[0,1]^2}\frac{e^{-(z_1-z_2)^2/4\lambda}}{\sqrt{2\pi \lambda}}\left\{\int_{\R}\frac{1}{\sqrt{2\pi \lambda}}\exp(-\frac{1}{\lambda}(v-\tfrac{z_1+z_2}{2})^2)dv\right\}dz_1dz_2\\
      &= \int_{[0,1]^2}\frac{1}{\sqrt{4\pi \lambda}}\exp(-\frac{(z_1-z_2)^2}{4\lambda})dz_1dz_2\\
      &= \int_{[0,1]}\left\{\int_{[0,1]}\frac{1}{\sqrt{2\pi (2\lambda)}}\exp(-\frac{(z_1-z_2)^2}{2(2\lambda)})dz_2\right\}dz_1\\
      &= \int_{[0,1]}\mathbb{P}[0\leq \sqrt{2\lambda}N+z_1\leq 1]dz_1,
\end{align*}
where $N$ is a standard Normal random variable. Then
\begin{align*}
    A &= \E\left[\int_{[0,1]}\chi_{[-\sqrt{2\lambda}N,1-\sqrt{2\lambda}N]}(z_1)dz_1\right] \\
      &= \E\left[(1-\sqrt{2\lambda}N)\chi_{[0,1]}(\sqrt{2\lambda}N)+(1+\sqrt{2\lambda}N)\chi_{[-1,0]}(\sqrt{2\lambda}N)\right]\\
      &= 2\E\left[(1-\sqrt{2\lambda}N)\chi_{[0,1]}(\sqrt{2\lambda}N)\right].
\end{align*}
The last equality follows by the symmetry of the distribution of $N$. Now replace $\lambda$ by its value, since $x, y\in [a,b]$ we obtain
\begin{align*}
A=2\E\left[(1-\frac{\sqrt{2t}}{x-y}N)\chi_{[0,\frac{x-y}{\sqrt{2t}}]}(N)\right]
&\leq 2\E\left[(1-\frac{\sqrt{2t}}{b-a}N)\chi_{[0,\frac{b-a}{\sqrt{2t}}]}(N)\right]\\
&\leq 2\mathbb{P}[0\leq N \leq\frac{\sqrt{2t}}{b-a}]<1.
\end{align*}
We then get $0\leq A<1$, and this finishes the proof of the lemma \ref{lemma1}.
\end{proof}
Consequently, we have
\begin{prop}
For all $t>0$ and $0\leq p<1$, we have
$$\int_a^b\int_a^b[\E(u(t,x)-u(t,y))^2]^{-(p+1)/2}dxdy<\infty .$$
\label{prop2}
\end{prop}
\begin{prop}
    For all $t>0$, there exists a square integrable version of the local time of $(u(t,x),\;x\in [a,b])$. We denote this version by $(L(\xi,y),\;y\in [a,b],\; \xi\in \R)$, where $L(\xi,y):=L(\xi,[a,y])$.
\end{prop}
\begin{proof}
    It is a consequence of Proposition \ref{prop2}, together with Theorem \ref{thmberman}.
\end{proof}
\subsubsection{Regularity of the local time}
In order to study the regularity of the local time, we need to recall the fundamental tool for that, the strong local nondeterminism concept (SLND). This notion was introduced by Cuzick and DuPreez in \cite{Cuzick} (see also \cite{Xiao1}), and used by many authors to investigate the law of iterated logarithm, Chung's law of the iterated logarithm, modulus of continuity for various Gaussian processes.

\begin{defn}
    Let $\{X_t,\;t\in I\}$ be a gaussian stochastic process with $0<\E(X_t^2)<\infty$ for any $t\in J$ where $J$ is a subinterval of $I$. Let $\phi$ be a function such that $\phi(0)=0$ and $\phi(r)>0$ for all $r>0$. Then $X$ is SLND on $J$ if there exist constants $K>0$ and $r_0>0$ such that for all $t\in J$ and all $0<r\leq \min\{|t|,r_0\}$,
    $$\var(X_t|X_s\;:\;s\in J,\; r\leq|s-t|\leq r_0 )\geq K \phi(r). $$
\end{defn}
\begin{thm}\label{SLND of u on x}
    For all $t>0$, there exists a positive constant $K=K(t,a,b)$, such that for all $0<r\leq |b-a|$, we have
    \begin{equation}
    \var\left(u(t,y)|u(t,x):x\in [a,b],\; r\leq|y-x|\leq |b-a|\right )\geq Kr.
    \label{VarCond}
    \end{equation}
\end{thm}
\begin{proof}
     It is enough to show that there exists a constant $K>0$ such that,
\begin{equation}
     \E\left(u(t,y)-\displaystyle\sum_{k=1}^n a_k u(t,x_k)\right)^2\geq Kr,
     \label{to prove}
     \end{equation}
for all integers $n\geq 1$, $(a_k)_{1}^{n}\in \R$ and $(x_k)_{1}^{n}\in [a,b]$ : $r\leq|y-x_k|\leq |b-a|$,
$\forall k\leq n$.\\
Parseval's identity implies
    \begin{equation}
\begin{split}
        &\E\left(u(t,y)-\displaystyle\sum_{k=1}^n a_k u(t,x_k)\right)^2\\
        &=\int_0^t\int_{\R}\left(G(t-s,y-z)-\displaystyle\sum_{k=1}^n a_kG(t-s,x_k-z)\right)^2dz ds\\
        &=\frac{1}{2\pi}\int_0^t\int_{\R}\left|\exp(iyu)-\sum_{k=1}^na_k\exp(ix_ku)\right|^2 \exp(-su^2)duds\\
        &=\frac{1}{2\pi}\int_{\R}\left|\exp(iyu)-\sum_{k=1}^na_k\exp(ix_ku)\right|^2 \frac{1-\exp(-tu^2)}{u^2}du:=Q(r).
        \label{E()2}
    \end{split}
\end{equation}
    So we just need to prove that $Q(r)\geq K r$.\\
  Let $\varphi:\R\to [0,1]$ be a function in $C^{\infty}(\R)$ such that $\varphi(0)=1$ and $supp(\varphi )\subset ]0,1[$. Denote by $\hat{\varphi}$ the Fourier transform of $\varphi$. Then $\hat{\varphi}\in C^{\infty}(\R)$ and $\hat{\varphi}(u)$ decays rapidly as $|u|\to \infty$.
    Set
    $$\varphi_r(\theta)=r^{-1}\varphi(r^{-1}\theta).$$
    By the inversion theorem we have
    \begin{equation}
    \varphi_r(\theta)=\frac{1}{2\pi}\int_{\R}e^{-iu\theta}\widehat{\varphi}(ru)du.
    \label{inversion}
    \end{equation}
    Since $r\leq|y-x_i|$ and $supp(\varphi)\subset ]0,1[$, we have $\varphi_r(y-x_i)=0$ for any $ k=1,...,n$. This and \eqref{inversion} imply that
    \begin{equation}
\begin{split}
        B   &:=\int_{\R}(\exp(iyu)-\sum_{k=1}^na_k\exp(ix_ku))\exp(-iyu)\widehat{\varphi}(ru)du\\
            &=2\pi(\varphi_r(0)-\sum_{k=1}^na_k\varphi_r(y-x_k))=2\pi r^{-1}.
            \label{B=}
        \end{split}
    \end{equation}
    On the other hand, by \eqref{E()2} and H\"{o}lder inequality, we obtain
    \begin{align*}
    B^2 &\leq \int_{\R}\left|\exp(iyu)-\sum_{k=1}^na_k\exp(ix_ku)\right|^2\frac{1-\exp(-tu^2)}{u^2}du\\
    &\qquad\times\int_{\R}\frac{u^2}{1-\exp(-tu^2)}|\widehat{\varphi}(ru)|^2du\\
        &=\E\left(u(t,y)-\displaystyle\sum_{k=1}^n a_k u(t,x_k)\right)^2\times\int_{\R}\frac{u^2}{1-\exp(-tu^2)}|\widehat{\varphi}(ru)|^2du\\
        &\leq \E\left(u(t,y)-\displaystyle\sum_{k=1}^n a_k u(t,x_k)\right)^2\frac{1}{r^3}\int_{\R}\frac{v^2}{1-\exp(-\frac{tv^2}{|b-a|^2})}|\widehat{\varphi}(v)|^2dv,
    \end{align*}
    where last inequality is justified by the change of variable $v=ru$ and $0<r\leq |b-a|$. So by \eqref{B=} we get
    $$4\pi^2 \frac{1}{r^2}\leq \E\left(u(t,y)-\displaystyle\sum_{k=1}^n a_k u(t,x_k)\right)^2\frac{1}{r^3}K,$$
    where
    $$K=\int_{\R}\frac{v^2}{1-\exp(-\frac{tv^2}{|b-a|^2})}|\widehat{\varphi}(v)|^2dv.$$
    Finally, \eqref{to prove} holds. This finishes the proof of Theorem \ref{SLND of u on x}.
\end{proof}
\begin{lem}\label{lem}
    Let $y,y+h\in [a, b]$. For any even positive integer $n$, we have
    \begin{equation}\label{Esperence in t}
      \E[L(\xi,y+h)-L(\xi,y)]^n\leq C_n |h|^{n/2},
    \end{equation}
    where $C_n$ is a positive constant.
\end{lem}
\begin{proof} For simplicity we will deal with $h>0$ such that $y+h\in [a,b]$. The other case uses the same calculation. Let $I=[y,y+h]$, then following \cite{GemanHoroviz} or \cite{Dozzi}, we have
  \begin{align*}
    \E[L(\xi,I)^n] &=(2\pi)^{-n}\int_{I^n}\int_{\R^n}e^{-i<\overline{u},\overline{\xi}>}\E\left[e^{i\sum_{k=1}^{n}u_ku(t,x_k)}\right]d\overline{u}d\overline{x}\\
                 &=(2\pi)^{-n}\int_{I^n}\int_{\R^n}e^{-i<\overline{u},\overline{\xi}>}e^{-\tfrac{1}{2}\var\left(\sum_{k=1}^{n}u_ku(t,x_k)\right)}d\overline{u}d\overline{x},
  \end{align*}
  where $\overline{\xi}=(\xi,\cdots,\xi)$ and $\overline{u}=(u_1,\cdots,u_n)$, hence
  \begin{equation}\label{E(L(x,I))}
    \E[L(\xi,I)^n]\leq (2\pi)^{-n}\int_{I^n}\int_{\R^n}e^{-\tfrac{1}{2}\var\left(\sum_{k=1}^{n}u_ku(t,x_k)\right)}d\overline{u}d\overline{x}.
  \end{equation}
On the other hand, for distinct $x_1,x_2,\cdots,x_n$, the matrix  $\cov(u(t,x_1),u(t,x_2),\cdots,u(t,x_n))$ is invertible. Then the following function is a gaussian density
    \begin{equation}\label{density}
      \frac{[\det\cov(u(t,x_1),u(t,x_2),\cdots,u(t,x_n))]^{1/2}}{(2\pi)^{n/2}}e^{-\tfrac{1}{2}\overline{u}\cov(u(t,x_1),u(t,x_2),\cdots,u(t,x_n))\overline{u}'},
    \end{equation}
    where $\overline{u}'$ denotes  the transpose of $\overline{u}$. Therefore
    \begin{equation}\label{int e^-1/2var =}
      \int_{\R^n}e^{-\tfrac{1}{2}\var\left(\sum_{k=1}^{n}u_ku(t,x_k)\right)}d\overline{u}=\frac{(2\pi)^{n/2}}{[\det\cov(u(t,x_1),u(t,x_2),\cdots,u(t,x_n))]^{1/2}}.
    \end{equation}
    Combining \eqref{E(L(x,I))} and \eqref{int e^-1/2var =}, we get
    \begin{equation}
      \E[L(\xi,I)^n]\leq (2\pi)^{-n/2}\int_{I^n}\frac{1}{[\det\cov(u(t,x_1),u(t,x_2),\cdots,u(t,x_n))]^{1/2}}d\overline{x}.
      \label{detcovint}
    \end{equation}
    It follows from (2.8) in \cite{Berman2} that
    \begin{equation}\label{det=var}
        \begin{split}
            \det & \cov(u(t,x_1),u(t,x_2),\cdots,u(t,x_n))\\
            &=\var(u(t,x_1))\prod_{j=2}^{n}\var(u(t,x_j)|u(t,x_1),\cdots,u(t,x_{j-1})).
        \end{split}
    \end{equation}
     \eqref{det=var} together with \eqref{VarCond} imply
    \begin{equation}\label{det geq}
      \det\cov(u(t,x_1),u(t,x_2),\cdots,u(t,x_n))\geq K^n |x_1-a| \prod_{j=2}^{n}\min_{1\leq i<j}|x_j-x_i|.
    \end{equation}
    By using \eqref{det geq} in \eqref{detcovint} we get
    \begin{align}
      \E[L(\xi,I)^n]&\leq C^n\int_{I^n}\frac{1}{|x_1-a|^{1/2}}\prod_{j=2}^{n}\frac{1}{\displaystyle\min_{1\leq i<j}|x_j-x_i|^{1/2}}d\overline{x}\nonumber\\
      &\leq C_{n} h^{n/2} ,
      \label{min}
    \end{align}
    where the last inequality is obtained by integrating in the order $dx_n,dx_{n-1},\cdots,dx_1$ and with the help of some elementary arguments.  This finishes the proof of the lemma \ref{lem}.
\end{proof}
\begin{lem}
  For all $\xi,\xi+k\in \R$, $y,y+h\in [a,b]$ and for all even positive integer $n$, there exists a constant $C_n>0$ such that
  $$\E[L(\xi+k,y+h)-L(\xi,y+h)-L(\xi+k,y)+L(\xi,y)]^n\leq C_n |k|^{n\delta}|h|^{n(1/2-\delta/2)},$$
  where $0<\delta<\frac{1}{2}.$
\end{lem}
\begin{proof}
  The proof uses the same techniques as those of Proposition \ref{prop 4}.
\end{proof}
   We can deduce by classical arguments (cf. Berman \cite[Theorem 8.1.]{Berman2} or Geman-J. Horowitz \cite[Theorem 26.1]{GemanHoroviz}) the following regularity result on the local time of the process $(u(t,x),\;x\in [a,b])$
  \begin{thm}\label{joint continuity in x}
    For any fixed $t>0$, the process $(u(t,x),\;x\in [a,b])$ has almost surely, a jointly continuous local time $(L(\xi, y), \;\xi\in\R ,  y\in [a,b])$. It satisfies a.s. a $\gamma$-H\"{o}lder condition in $y$, uniformly in $\xi$, for every $\gamma<\tfrac{1}{2}$ : there exist random variables $\eta$ and $\eta'$ which are almost surely positive and finite such that
    \begin{equation}\label{holder conddition in y}
      \sup_{\xi}|L(\xi,y+h)-L(\xi,y)|\leq \eta'|h|^{\gamma},
    \end{equation}
    for all $y$, $y+h\in [a,b]$ and all $|h|<\eta$.
  \end{thm}
   We also have, by \cite[Lemma 2.1.]{BLD}, the following Besov regularity of the local time $L(\xi,y)$ in the space variable $y$
\begin{thm}
  For all $\lambda>0$ and $p>\frac{1}{\lambda}$,
  $$\mathbb{P}\left(L(\xi,.)\in \mathcal{B}^{\omega_{\lambda}}_p\right)=1,$$
  where $\omega_{\lambda}(t)=t^{1/2}(\log(1/t))^{\lambda}$ and $L(\xi,.)$ is the sample paths $y\to L(\xi,y)$, $y\in [0,1]$.
\end{thm}
For fixed $t>0$, let
$M_t(\xi)=\{x\in [a,b]\;:\;u(t,x)=\xi\} $
be the $\xi$-level set of the process $ (u(t,x)\,,\, x\in [a,b]) $. Proceeding in the same way as for
$ (u(t,x)\,,\, t\in [0,T]) $, we have
\begin{cor}
  For all $t>0$, we have for almost every $x_0$
  $$\dim_H(M_t(u(t,x_0)))= \frac{1}{2}\quad a.s.$$
\end{cor}

\section*{Acknowledgement(s)}
The first author would like to warmly thank Professor M. Dozzi for his fruitful discussions on an earlier version of this article.


\begin{thebibliography}{99}
    \bibitem{adler}
    R.J. Adler, \emph{The Geometry of Random Fields}, Wiley, New York, 1981.
    \bibitem{ArayaTudor}
    H. Araya and C. Tudor, \emph{Asymptotic expansion for the quadratic variations
of the solution to the heat equation
with additive white noise}, Stochastics and Dynamics (2020). DOI: 10.1142/S0219493721500106
     \bibitem{AyaWuXiao}
    A. Ayache, D. Wu, and Y. Xiao, \emph{Joint continuity of the local times of Fractional Brownian sheet}, Ann. Inst. H. Poincar\'e Probab. Statist. 44(4) (2005), pp. 727--748.
     \bibitem{BalanTudor1}
    R. Balan and C.A. Tudor, \emph{The stochastic heat equation with fractional colored noise: Existence of the solution}, Latin. Amer. J. Probab. Math. Stat. 4 (2008), pp. 57--87.
    \bibitem{Berman1}
    S.M. Berman, \emph{Harmonic analysis of local times and sample functions of Gaussian processes}, Trans. Amer. Math. Soc. 143 (1969), pp. 269--281.
    \bibitem{Berman2}
    S.M. Berman, \emph{Local nondeterminism and local times of Gaussian processes}, Indiana Univ. Math. J.  23 (1974), pp. 69--94.
    \bibitem{Boufoussi}
    B. Boufoussi, \emph{Régularité du temps local brownien dans les espaces de Besov-Orlicz},
    Studia. Math. 118(2) (1996), pp. 145--156.
    \bibitem{BDG}
    B. Boufoussi, M. Dozzi, and R. Guerbaz, \emph{Regularity of the local times of the multifractional Brownian motion}, Stochastics 78 (2008), pp. 33--49.
    \bibitem{BLD}
   B. Boufoussi, E. Lakhel, and M. Dozzi,  \emph{A Kolmogorov and tightness criterion in modular Besov
   spaces and an application to a class of Gaussian processes}, Stoch. Anal. Appl. 23 (2005), pp. 665--685.
    \bibitem{BoufoussiR}
    B. Boufoussi and B. Roynette, \emph{Le temps local brownien appartient p.s \`a l'espace de Besov $\mathcal{B}^{1/2}_{p,\infty}$}, C.R.A.S. Paris 1316 (1993), pp. 843--848.
    \bibitem{Carmona}
    R. A. Carmona and B. L. Rozovskii, \emph{Stochastic partial differential equations: six perspectives} Vol. 64,  Mathematical surveys and monographs, AMS, 1999.
    \bibitem{ChenDalang}
    L. Chen and R.C. Dalang \emph{Moments and growth indices for the nonlinear stochastic heat equation with rough initial conditions}, The Ann. Probab. 43(6) (2015), pp. 3006--3051.
    \bibitem{Ciesielski}
    Z. Ciesielski, \emph{Orlicz spaces, spline systems, and Brownian motion}, Constr. Approx. 9 (1993), pp. 191--208.
     \bibitem{Roynette}
Z. Ciesielski, G. Keryacharian, and B. Roynette, \emph{Quelques espaces fonctionnels associ\'es
$\grave{a}$ des processus gaussiens}, Studia Matheamatica 107(2) (1993), pp. 171--204.
     \bibitem{CONUS}
     D. Conus, M. Joseph, D. Khoshnevisan, and S. Y. Shiu, \emph{Initial measures for the stochastic heat equation}, Ann. Inst. Henri Poincar\'e  Probab. Stat. 50 (2014), pp. 136--153.
    \bibitem{Cioica1}
    P.A. Cioica, \emph{Besov Regularity of Stochastic Partial Differential Equations on
Bounded Lipschitz Domains}, Dissertation (2014), Philipps-Universität Marburg, Logos, Berlin, 2015.
    \bibitem{Cioica2}
    P.A. Cioica, K-H. Kim, K. Lee, and F. Lindner, \emph{On the $L_{q}(L_{p})$-regularity and Besov smoothness of stochastic parabolic equations on bounded Lipschitz
domains}, Electron. J. Probab. 18(82) (2013), pp. 1--41.
    \bibitem{Cioica3}
    P.A. Cioica, S. Dahlke, S. Kinzel, F. Lindner, T. Raasch, K.
Ritter, and R.L. Schilling, \emph{Spatial Besov regularity for stochastic partial differential
equations on Lipschitz domains}, Studia Math. 207(3) (2011), pp. 97--234.
    \bibitem{Cuzick}
    J. Cuzick and J.P. DuPreez, \emph{Joint continuity of Gaussian local times}, Ann. Probability 10 (1982),  pp. 810--817.
    \bibitem{Dalang}
    R.C. Dalang, \emph{Extending the martingale measure stochastic integral with applications to spatially homogeneous s.p.d.e.'s}, Electr. J. Probab. 4(6) (1999), 29 pp.
    \bibitem{DalangKHOSHNEVISAN}
    R.C. Dalang, D. Khoshnevisan, C. Mueller, D. Nualart, and Y. Xiao,  \emph{A Minicourse on Stochastic Partial Differential Equations}, Lecture Notes in Math. Vol.~1962, Springer, Berlin, 2009.
    \bibitem{DalangSole}
    R.C. Dalang and M. Sanz-Sol\'e,  \emph{Regularity of the sample paths of a class of second-order spde's},  J. Funct. Anal. 227 (2005), pp. 304--337.
     \bibitem{Da Prato}
     G. Da Prato and J. Zabczyk, \emph{Stochastic Equations in Infinite Dimensions}, Encyclopedia Math. Appl. vol.~45, Cambridge Univ. Press, Cambridge, 1998.
    \bibitem{Deaconu}
    M. Deaconu, \emph{Processus stochastiques et \'equations aux d\'eriv\'ees partielles. Applications des espaces de Besov aux processus stochastiques}, Thesis, Institut Elie Cartan Nancy, 1997.
     \bibitem{Denk}
G. Denk, D. Meintrup, and S. Schaffer, \emph{ Modeling, simulation and optimization
of integrated circuits}, Intern. Ser. Numerical Math. 146 (2004), pp. 251--267.
    \bibitem{Dozzi}
    M. Dozzi, \emph{Occupation density and sample path properties of N-parameter processes}, Topics in Spatial Stochastic Processes. Lecture Notes in Math. Springer, Berlin 1802 (2003), pp. 127--166.
    \bibitem{Falconer}
    K. Falconer, \emph{Fractal geometry: mathematical foundations and applications}, Wiley, 2003.
    \bibitem{GemanHoroviz}
    D. Geman and J. Horowitz, \emph{Occupation densities}, Ann. Probability 8 (1980), pp. 1--67.
     \bibitem{Verrar2}
    T.P. Hyt\"{o}nen and M.C. Veraar, \emph{On Besov regularity of Brownian motions in infinite dimensions}, Probab. Math. Statist. 28(1) (2008), pp. 143--162.
    \bibitem{Kou}
    S.C. Kou, and X.S. Xie, \emph{Generalized Langevin equation with fractional
Gaussian noise: subdiffusion within a single protein molecule}, Phys. Rev. Letters 93(18) (2004), American Phys. Soc., 180603.
    \bibitem{Krylov 1977}
    N.V. Krylov and B.L. Rozovskii, \emph{On the Cauchy problem for linear stochastic partial differential equations},
Isz. Akad. Nauk SSSR Ser. Mat. 41(6) (1977), pp. 1329--1347.
\bibitem{Krylov 1982a}
N.V. Krylov and B.L. Rozovskii, \emph{On characteristics of the degenerate parabolic Ito equations of the
second order}, Proc. Petrovskii Sem. 8 (1982), pp. 153--168.
\bibitem{Krylov 1982b}
N.V. Krylov and B.L. Rozovskii, \emph{Stochastic partial differential equations and diffusion processes}, Uspekhi
Mat. Nauk 37(6) (1982), pp. 75--95.
    \bibitem{Mueller}
    C. Mueller, \emph{On the support of solutions to the heat equation with noise}, Stochastics and Stoch. Reports 37 (1991), pp. 225--245.
     \bibitem{Verrar1}
    M. Ondreját and  M. Veraar, \emph{On temporal regularity of stochastic convolutions in 2-smooth Banach spaces}, Ann. Inst. H. Poincaré Probab. Statist. 56(3) (2020), pp. 1792--1808.
    \bibitem{Ouahhabi}
    H. Ouahhabi and C.A. Tudor, \emph{Additive functionals of the solution to fractional stochastic heat equation}, Journal of Fourier Analysis and Applications 19(4) (2013), pp. 777--791.
    \bibitem{Pitt}
    L.D. Pitt, \emph{Local times for Gaussian vector fields}, Indiana Univ. Math. J. 27 (1978), pp. 309--330.
    \bibitem{RoynetteMB}
    B. Roynette,  \emph{Mouvement Brownien et espaces de Besov}, Stochastics and Stoch. Reports 43 (1993), pp. 221--260.
     \bibitem{SoleSara1}
    M. Sanz-Sol\'e and M. Sarr\`a, \emph{Path properties of a class of Gaussian processes with applications to spde's}, Canadian Mathematical
Society Conference Proceedings 28 (2000),  pp. 303--316.
    \bibitem{SoleSara}
    M. Sanz-Sol\'e and M. Sarr\`a, \emph{H\"{o}lder continuity for the stochastic heat equation with spatially correlated noise}, In Seminar on Stochastic Analysis, Random Fields and Applications. Progress in
Probability,  Birkh\"{a}user, Basel 52 (2002), pp. 259--268.
    \bibitem{Sw}
    J. Swanson, \emph{Variations of the solution to a stochastic heat equations}, Ann. Probab. 15(6) (2007), pp. 2122--2159.
    \bibitem{Tribel}
     H. Triebel, \emph{Interpolation theory, function spaces, differential operators}, second ed., Johann Ambrosius Barth, Heidelberg, 1995.
    \bibitem{Tudor}
    C.A. Tudor and  Y. Xiao, \emph{Sample paths of the solution to the fractional-colored stochastic heat equation}, Stoch. Dyn. 27 (2017), 20 pp.
     \bibitem{TudorXiao}
    C.A. Tudor and  Y. Xiao, \emph{Sample paths properties of bifractional Brownian motion}, Bernoulli 13(4) (2007), pp. 1023--1052.
 \bibitem{Verrar3}
 M.C. Veraar, \emph{Correlation inequalities and applications to vector-valued Gaussian random
variables and fractional Brownian motion}, Potential Anal. 30(4) (2009), pp. 341--370.
    \bibitem{Walsh}
    J.B. Walsh, \emph{An introduction to stochastic partial differential equations}, In \'Ecole d'\'Et\'e de Probabilit\'es de Saint Flour XIV, Lecture Notes in Mathematics, Springer-Verlag, 1180 (1986), pp. 266--439.
      \bibitem{XiaYan}
D. Xia and L. Yan, \emph{Some properties of the solution to
fractional heat equation with a fractional
Brownian noise}, Advances in differential equations 107 (2017).
    \bibitem{Xiao2}
    Y. Xiao, \emph{H\"{o}lder conditions for the local times and the Hausdorff measure of the level sets of Gaussian random fields}, Probab. Th. Rel. Fields 109 (1997), pp. 129--157.
    \bibitem{Xiao1}
    Y. Xiao, \emph{Properties of local nondeterminism of Gaussian and stable random fields and their applications},  Ann. Fac. Sci. Toulouse Math. XV (2006), pp. 157--193.
\bibitem{Xiao3}
    Y. Xiao, \emph{Strong local nondeterminism of Gaussian random fields and its applications}, in: Asymptotic Theory in Probability and Statistics with Applications, (T.-L. Lai, Q.-M. Shao
and L. Qian, editors), Higher Education Press, Beijing, pp. 136--176.
   \bibitem{Xiao4}
    Y. Xiao, \emph{Sample path properties of anisotropic Gaussian random fields}, in: A Minicourse on Stochastic
Partial Differential Equations. Lecture Notes in Math., Springer, Berlin, vol. 1962, (2009), pp. 145--212.
     \bibitem{Xiao5}
     Y. Xiao, \emph{Sharp space-time regularity of the solution to stochastic
heat equation driven by fractional-colored noise}, Stochastic Analysis and Applications, (2020). DOI: 10.1080/07362994.2020.1721301
 \end{thebibliography}
\end{document}